\pdfoutput=1

\documentclass{article}

\usepackage{graphicx}
\usepackage{amsmath}
\usepackage{mystuff_new}
\RequirePackage[capitalize,sort&compress]{cleveref}
\crefname{assumption}{Assumption}{Assumptions}
\crefname{corollary}{Corollary}{Corollaries}
\usepackage[cmyk]{xcolor}
\crefname{assumption}{Assumption}{Assumptions}
\crefname{corollary}{Corollary}{Corollaries}

\newcommand\Ito{It\^o\xspace}
\newcommand\EM{Euler--Maruyama\xspace}
\usepackage{enumerate}
\usepackage{eucal}
\delimstyle{3pt}{850}
\def\method{\vec S}    
                             
\def\sdeconst{L_{\operatorname{SDE}}}
\newcommand{\shat}{\skew{2}\hat{s}}
\newcommand{\rhat}{\skew{2}\hat{r}}
\newcommand{\uhat}{\skew{2}\hat{u}}
\newcommand{\yhat}{\skew{1}\hat{y}}
\newcommand{\Cb}{\mathrm{C}_\mathrm{b}}
\newcommand{\Kbar}{\bar{K}}
\newcommand{\Ybar}{\bar{Y}}
\newcommand{\lbar}[1]{\mkern
  2mu\overline{\mkern-2mu#1\mkern-2mu}\mkern 2mu}
\renewcommand{\Kbar}{\skew{3}\lbar{{K}}}
\renewcommand{\Ybar}{\skew{1}\lbar{{Y}}}
\newcommand{\probs}{\mathbb{P}}
\newcommand\sdeconstd{K_{\operatorname{SDE}}}
\newcommand\newxconst{K_{\operatorname{TES}}}
\newcommand\xconst{K_{\operatorname{TE}}}
\newcommand\fixedconst{K_{\operatorname{F}}}
\newcommand\adapconst{K_{\operatorname{A}}}
\newcommand\KBI{K_{\operatorname{BI}}}

\crefname{ineq}{}{}
\creflabelformat{ineq}{(#2#1#3)}
\renewcommand{\vec}[1]{{#1}} 
\begin{document}

\title{On the pathwise approximation 
  of stochastic differential equations
}

\author{ Tony Shardlow\footnote{University of Bath \texttt{t.shardlow@bath.ac.uk}}\and%
   Phillip Taylor }
 
\maketitle

\begin{abstract}
We consider one-step methods for integrating stochastic
    differential equations and prove pathwise convergence using
    ideas from rough path theory. In contrast to alternative
    theories of pathwise convergence, no knowledge is required of
    convergence in $p$th mean and the analysis starts from a
    pathwise bound on the sum of the truncation errors. We show
    how the theory is applied to the Euler--Maruyama method with
    fixed and adaptive time-stepping strategies. The assumption
    on the truncation errors suggests an error-control strategy
    and we implement this as an adaptive time-stepping
    Euler--Maruyama method using bounded diffusions. We prove the
    adaptive method converges and show some computational
    experiments.
\end{abstract}

\section{Introduction}

Let $(\Omega, \mathcal{F}, \mathcal{F}_t, \probs)$ be a filtered
probability space and consider independent
$\mathcal{F}_t$-Brownian motions $W^j(t)$ for $j=1,\dots,m$.  We
study the following \Ito stochastic differential equation (SDE)
in $\real^d$:
\begin{equation}
  \label{eq:sode}
  d\vec y(t)%
  = \vec g_0(\vec y(t))\,dt%
  + \sum_{j=1}^m \vec g_j(\vec y(t))\,dW^j(t)
\end{equation}
where $\vec g_j\colon\real^d\to\real^d$ for
$j=0,\dots,m$. We assume that $\vec g_j$ are sufficiently
regular and there exists a stochastic process $\vec y(t)$
that satisfies this equation on a time interval $[0,T]$ and,
if an initial condition is specified, the solution is unique
(in the pathwise sense).  We denote by $\vec y(t; s, \vec
z)$ the solution of~\cref{eq:sode} for $t\in [s,T]$ with
initial condition $\vec y(s) = \vec z\in\real^d$.  In
general, exact solutions $\vec y(t)$ are not known and
numerical integrators are required to determine quantities
of interest, such as averages, sample paths, or exit
times. In this paper, we look at one-step methods for
approximating sample paths of $\vec y(t)$ and analyse the
pathwise error using techniques from rough path theory
\citep{MR2387018,friz_multidimensional_2010}.  In dynamical
system, we are often interested in how sample paths of SDEs
change with model parameters and it is important to compute
sample paths reliably.  The main result is
\Cref{thm:global_alt}. It gives pathwise convergence of the
one-step method at a polynomial rate subject to a regularity
condition on the sample paths (\Cref{ass:4.1,ass:4.2}) and a
bound on the sum of the truncation errors
(\Cref{ass:newx}). As well as identifying the rate of
convergence in terms of the bound on the truncation-error
sum, we identify the constant explicitly in terms of those
appearing in the
assumptions.

 Pathwise error analysis
is normally performed \citep{MR1625576,MR2320830} by
showing the $p$th mean error converges at a polynomial rate and applying the
Borel--Cantelli lemma. \cref{thm:global_alt} predicts the same
rates of convergence, for example, for the fixed time-stepping
Euler--Maruyama or Milstein method. However,  it does not
use $p$th mean error estimates.

We take particular interest in the so called
\emph{bounded-diffusion} time-stepping strategy
of~\citep{MR1722281}. Instead of taking uniformly spaced
times and sampling the Brownian increments from the Gaussian
distribution, we choose random times such that the Brownian
increment is bounded. Specifically, we define a cuboid
$[0,a_0]\times[-a_1,a_1]\times\dots\times[-a_m,a_m]$ and
choose the first exit time $\tau$ of the process $(t,
W^1(t),\dots,W^m(t))$ from the cuboid. This defines a
stopping time $\tau$ and associated exit points $W^i(\tau)$
that can be used for the time step and Brownian increments
in a numerical integrator for \cref{eq:sode}. We will use
the convergence criterion developed for
\Cref{thm:global_alt} to choose $a_i$ adaptively and thereby
implement an error-control strategy for the Euler--Maruyama
method.  The adaptivity leads to an improvement in the
constant in the theoretical error bound, compared to fixed
time-stepping. The constant depends on the inherent
exponential divergence of sample paths of the SDE with
different initial data (see~\cref{ass:4.2}) and the constant
in the local truncation error (see~\cref{ass:newx}). The
adaptive %
strategy is able to control the second source of error, not
the %
first.


The paper is organised as follows. \cref{sec:bg} gives background
on the time-stepping methods of interest. Working pathwise from
the start, \cref{sec:anal} provides the statement and proof of
the main result \Cref{thm:global_alt}.  In \cref{sec:em}, we give
preliminary lemmas that provide pathwise bounds on the sum of the
truncation errors, which help in establishing \cref{ass:newx},
and show that the pathwise-convergence theorem applies to the
Euler--Maruyama method with fixed time-steps. In \cref{sec:bd},
we introduce two adaptive time-stepping strategies based on
bounded diffusions and present convergence theory and numerical
experiments. An appendix reviews some useful results.

\section{Background}%
\label{sec:bg}      %

Our pathwise convergence theory applies to one-step methods
for \cref{eq:sode} in the case of variable and random time-steps. We work on the time interval $[0,T]$ and consider partitions
$\mathcal{T}$ of $[0,T]$.
 
\begin{definition}[partitions] Let $\mathcal{T}$ denote the set
  of partitions $0=\tau_0<\tau_{1}<\dots<\tau_{N}=T$ consisting
  of $[0,T]$-valued random variables.  Let
  $\mathcal{T}_{\operatorname{stop}}$ be the subset of
  $\mathcal{T}$ consisting of stopping times (i.e., if
  $(\tau_0,\dots, \tau_{N})\in \mathcal{T}_{\operatorname{stop}}$
  then $\tau_j$ is $\mathcal{F}_{\tau_j}$-measurable for
  $j=0,\dots,N$). Let $\shat=\tau_n$ if $\tau_n\le
  s<\tau_{n+1}$.
\end{definition}
We generate approximations $\vec y_k=S_{0k}(z)$ to $\vec
y(\tau_k;\tau_0,\vec z)$ at times $\tau_k$ in a partition
$(\tau_0,\dots,\tau_n)\in \mathcal{T}$ using one-step methods.

\begin{definition}[one-step method]   
  Given $\tau_n<\tau_{n+1}$, a \emph{one-step method} $\vec
  S_{n,n+1}$ is a map from $\real^d$ to the set of
  $\real^d$-valued random variables. For a given partition
  $\mathcal{T}$ and $0\le k<n\le N$, we use the notation
  $\method((\tau_k,\dots,\tau_n),\vec z)$ or the
  abbreviation $\method_{kn}(\vec z)$ to denote $\vec
  S_{n-1,n}\circ\dots \circ\vec S_{k,k+1}(\vec z)$, the
  action of applying the one-step method successively over
  the time steps
  $[\tau_k,\tau_{k+1}],\dots,[\tau_{n-1},\tau_n]$.
\end{definition}
The simplest useful example is the Euler--Maruyama  method with fixed
time-step $h$ given by $\method_{n,n+1}(\vec y_n)=\vec y_{n+1}$ for 
 times $\tau_n=\tau_0+n h$ and
 \begin{equation}
   \label{eq:em}
   \vec y_{n+1}%
   =\vec y_n%
   + \vec g_0(\vec y_n)(\tau_{n+1}-\tau_n)%
   + \sum_{j=1}^m \vec g_j(\vec y_n) \pp{W^j(\tau_{n+1})-W^j(\tau_n)}.
 \end{equation}  The implicit {\EM} method, given by
 \begin{equation}
  \label{eq:emi}
  \vec y_{n+1}%
  =\vec y_n%
  + \vec g_0(\vec y_{n+1})(\tau_{n+1}-\tau_n)%
  + \sum_{j=1}^m \vec g_j(\vec y_n) \pp{W^j(\tau_{n+1}) -W^j(\tau_n)},
\end{equation}
is included if the nonlinear equations can be solved for any
$y_n$ to define $y_{n+1}$ uniquely. 
We will introduce an example of random times $\tau_n$ in
\cref{sec:bd}.

Key to the analysis of convergence  of one-step methods is the
local truncation error.         %

\begin{definition}[local truncation error]%
  \label[definition]{def:lte}%
  For $\tau_n<\tau_{n+1}$, the \emph{local truncation error} at $\vec
  z\in\real^d$ of a one-step method $S_{n,n+1}(z)$ is
  \begin{equation}%
    \label{eq:lte}
    \vec  \delta(\tau_n, \tau_{n+1},\vec z)%
     \coloneq    \vec y(\tau_{n+1};\tau_n,\vec z)%
    -\method_{n,n+1}(\vec z).%
  \end{equation}
\end{definition} 
For the {\EM} method,
  writing $dW^0(t)=dt$, the local truncation error is
\begin{gather}%
  \begin{split}\label{eq:lte_em}
    \delta(\tau_k, \tau_{k+1}, z)%
    =&\sum_{i,j=0}^m \int_{\tau_k}^{\tau_{k+1}}%
    \int_{\tau_k}^s%
    q_{ij}\pp{y(r;\tau_k,z)} \,dW^i(r) \,dW^j(s),
  \end{split}%
\end{gather}
where
\begin{equation}
  q_{ij}(y)%
   \coloneq   \begin{cases}\displaystyle
    Dg_j(y) g_0(y)%
    +\frac 12 \sum_{k=1}^m D^2 g_j(y) (g_k(y), g_k(y)),& i=0,\\
    Dg_j(y) g_i(y),& i\ne 0.
  \end{cases}\label{eq:q}
\end{equation}
See for example \citep{kloeden:1992,milstein:1995}.
Under regularity assumptions on $g_j$, we can estimate the $p$th
moment by using the Burkholder--Davis--Gundy inequality and find
$\mean{\norm{\delta}^p}=\order{h^{p} }$ for $p\ge 2$ and then
show that the local truncation error $\delta(\tau_k,\tau_{k+1},z)=\order{h^{\gamma+1/2-\epsilon}}$
with $\gamma=1/2$ and $\epsilon>0$. 
  Note the use of $\epsilon$ in writing the condition on the
  truncation error, similar to saying Brownian motion is
  H\"older continuous with exponent %
  $1/2-\epsilon$, any $\epsilon>0$.   
We make the following assumption on the truncation error.

\begin{assumption}\label[assumption]{ass:bluenew}
  For a partition  $(\tau_0,\dots,\tau_N)\in \mathcal{T}$,
an  $\epsilon\in(0,1)$, and a random variable $\xconst>0$, it holds that
  \begin{equation}
    \label[ineq]{eq:3} 
    \norm{\vec \delta(\tau_k, \tau_{k+1}, \vec z)}%
    \le \xconst |\tau_{k+1}-\tau_k|^{\gamma+1/2-\epsilon}, %
    \qquad \text{ $\forall\vec z\in\real^d$, $k=0,\dots,N-1$.}
  \end{equation}
  
  \end{assumption}

In general, we want to refine the partition $\mathcal{T}$
and keep $\xconst$ and $\epsilon$ constant in \cref{eq:3},
in which case $\gamma$ indicates the pathwise order of
convergence  with respect to the mesh width
$h \coloneq \max_k|\tau_{k+1}-\tau_k|$ (as we show in
\cref{thm:global_alt}).  We will denote such methods by
$\vec S^\gamma_{kn}$.
Before \cref{thm:global_alt}, we review a simpler result
from~\citep{MR1470933} that gives order $\gamma-1/2$
convergence. We assume the following continuity with respect to
initial data property, which is satisfied by $y(t;s,z)$ if
$g_j\in \Cb^3(\real^d,\real^d)$ \citep[Theorem
10.26]{friz_multidimensional_2010}. Here,
$C_{\mathrm b}^r(\real^d,\real^d)$ is the set of functions $g\colon
\real^d\to \real^d$ with $r\in\naturals$ uniformly bounded and
continuous derivatives and norm
$\norm{g}_{C_{\mathrm b}^r}\coloneq\sup_{|\alpha|\le r}\sup_{\vec x\in
  \real^d}\norm{D^\alpha g(\vec x)}$.

\begin{assumption}%
  \label[assumption]{ass:4.1}%
  For a random variable $\sdeconst>0$,
  \begin{equation}\label[ineq]{eq:boom}
  \norm{ \vec y(t; s, \vec z^1)-\vec y(t; s, \vec z^2) } %
  \le \sdeconst\norm{\vec z^1-\vec z^2}
  \end{equation}
for  $0\le s\le t\le T$ and $\vec z^1,\vec z^2\in\real^d$.
\end{assumption}
We define $h \coloneq \max_k|\tau_{k+1}-\tau_k|$ and will use $h$ to
measure rates of convergence.

\begin{theorem}[convergence rate $\gamma-1/2$]
  \label[theorem]{thm:global1}%
  Let \cref{ass:4.1} hold for \cref{eq:sode} and
  \cref{ass:bluenew} hold for $\vec S^\gamma_{kn}$. Then,
  for any  $\vec y_0\in\real^d$,
  \[
  \max_{n=0,\dots,N}
  \norm{\method^\gamma_{0n}(\vec y_0)-\vec y(\tau_n;0,\vec y_0)}%
  \le \sdeconst \,\xconst \,T\, h ^{\gamma-\epsilon-1/2},
  \]
  where $\epsilon$, $\xconst$, and $\sdeconst$ are given in
  \cref{eq:boom,eq:3}.
\end{theorem}

\begin{proof} Let ${\vec y}_n \coloneq 
  \method_{0n}^\gamma(\vec y_0)$ and let $\vec
  y^k(t) \coloneq \vec y(t; \tau_k, {\vec y}_k)$ for
  $t\in[\tau_k,T]$. Then,
  \begin{align*}
    \norm{\vec y^k(\tau_{k+1})-{\vec y}_{k+1}}%
    = \norm{\vec y(\tau_{k+1}; \tau_k, {\vec y}_k)%
      - \method^\gamma_{k,k+1}({\vec y}_k)}.
  \end{align*}
  By \cref{def:lte}, this is bounded by
  $\xconst|\tau_{k+1}-\tau_k|^{\gamma+1/2-\epsilon}$. By
  \cref{ass:4.1},
  \begin{align*}
    \norm{\vec y^k(\tau_n)-\vec y^{k+1}(\tau_n)}%
    &\le\sdeconst\norm{\vec y^k(\tau_{k+1})-\vec
      y^{k+1}(\tau_{k+1})}%
    = \sdeconst\norm{\vec y^k(\tau_{k+1})-{\vec y}_{k+1}} \\
    &\le
    \sdeconst\pp{\xconst|\tau_{k+1}-\tau_k|^{\gamma-\epsilon+1/2}}.
  \end{align*}
  If we now sum over $k=0,\dots,n-1$, the triangle inequality applies
  and we obtain
  \[
  \norm{\vec y(\tau_n; 0, \vec y_0)-\method^\gamma_{0n}(\vec y_0)}%
  = \norm{\vec y^0(\tau_n)-\vec y^n(\tau_n)}%
  \leq \sdeconst\, \xconst\sum_{k=0}^{n-1}
  \abs{\tau_{k+1}-\tau_k}^{\gamma-\epsilon+1/2},
  \]
  since $\vec y^0(\tau_n) =\vec y(\tau_n; 0, \vec y_0)$ and $\vec
  y^n(\tau_n) = \method^\gamma_{0n}(\vec y_0)$. Finally, $\tau_n\le
  T$, so that
  \[
  \norm{\vec y(\tau_n; 0,\vec y_0)-\method^\gamma_{0n}(\vec y_0)}%
  \leq \sdeconst \,\xconst \,T\,h^{\gamma-\epsilon-1/2}.
\]
\end{proof}
Suppose that $\mathcal{T}_h$ is a family of partitions with
$h=\max_k\abs{\tau_{k+1}-\tau_k}$ for
$(\tau_0,\dots,\tau_N)\in \mathcal{T}_h$.  When
\cref{ass:4.1} holds with the same $\xconst$ for each
$\mathcal{T}_h$, this theorem provides convergence of the
numerical approximation $S^\gamma_{0n}(y_0)$ on the
partition $\mathcal{T}_h$ to the true solution in the
pathwise sense in the limit $h\to 0$ and the pathwise error
is $\order{h^{\gamma-\epsilon-1/2}}$, for $\epsilon>0$. The
conditions are general (e.g., allowing time steps that are
not uniformly spaced or are not stopping times). However the
theorem does not imply convergence of the {\EM} method and
is not optimal for the Milstein method where the error is
$\order{h^{1-\epsilon}}$ not $\order{h^{1/2-\epsilon}}$ for
fixed time-steps.

\section{Main result on pathwise convergence}%
\label{sec:anal}%
To achieve a rate of convergence higher than the rate in
\cref{thm:global1}, we introduce two more assumptions. The first like
\cref{ass:4.1} concerns the solution of the SDE itself. Again, the
condition holds if $g_j\in \Cb^3(\real^d,\real^d)$ (see \cref{lemma:quote_me}).

\begin{assumption}%
  \label[assumption]{ass:4.2}%
  For an $0<\epsilon<1/2$ and a random variable
  $\sdeconstd>0$,
  \begin{align}
    \norm{\pp{\vec y(t; s, \vec z^1)-\vec z^1}%
      - \pp{\vec y(t; s,\vec z^2)-\vec z^2}}%
    &\le \sdeconstd(t-s)^{1/2-\epsilon/3} \norm{\vec z^1-\vec z^2},%
    \label[ineq]{eq:lte2}
  \end{align}
for all $0\le s<t\le T$ and  $\vec z^1,\vec z^2\in\real^d$.
\end{assumption}

The second assumption is a pathwise version of the
independent-increment property of Brownian motion and sets a
 bound on the sum of truncation errors along each path. We will
verify this assumption in the case $(\tau_0,\dots,\tau_N)\in
{\cal T}_{\operatorname{stop}}\subset {\cal T}$, so that the
$\tau_j$ are stopping times. Indeed, \citep{MR1470933}
provides an example of a time-stepping strategy where the
{\EM} method fails when $\tau_j$ are not adapted. 
 Consider an integrator $\vec S^\gamma_{kn}$ with respect to
  $(\tau_0,\dots,\tau_{N})\in \mathcal{T}$ and recall
  $h \coloneq \max_{k=0,\dots,N-1}|\tau_{k+1}-\tau_k|$. 
\begin{assumption}[truncation-error sum]
  \label[assumption]{ass:newx}%
     For an initial
  condition $y_0\in\real^d$, an $\epsilon\in (0,1/2)$, and
  random variable $\newxconst>0$, it holds that
  \begin{equation}
    \label[ineq]{eq:2x}
    \norm{\vec X_{kn}(\vec y_0)}%
    \le \newxconst (\tau_n-\tau_k)^{(1+\epsilon)/2}
    h^{\gamma-\epsilon},\quad%
0\le k<n\le N,
  \end{equation}
  where $\vec X_{kn}(y_0) \coloneq \sum_{j=k}^{n-1} \vec \delta(\tau_j,
  \tau_{j+1}, \vec y_j)$ and $\vec y_n=\method_{0n}^\gamma(\vec
  y_0)$. 
\end{assumption}
In general, we want to refine the partition
  $\mathcal{T}$ and take limits as $h\downarrow 0$ while
  keeping  $\newxconst$  and $\epsilon$ constant. 

For  the {\EM} method~\eqref{eq:em} 
with fixed time-step $\tau_{n+1}-\tau_n=h$,
\cref{eq:2x} follows because 
\[
X_{kn}(y_0)%
=X(\tau_k,\tau_n),\quad%
\text{ for }%
X(r,t)%
 \coloneq \sum_{i,j=0}^m\int_{r}^{t} \int_{\shat}^{ s} %
q_{ij}\pp{ y(r; \rhat, \yhat(r))} \,dW^i(r) \,dW^j(s),
\] 
by writing \cref{eq:lte_em,eq:q} with the notation
$\yhat(r) \coloneq y_n$ for $\tau_n\le r <\tau_{n+1}$. Then, if
the $q_{ij}$ are well-behaved, we can show that
$\mean{\norm{X(r,t)}^p}=\order{h^{p/2} (t-r)^{p/2}}$ and we
expect that $\norm{X_{kn}(y_0)}\le \newxconst
(\tau_n-\tau_k)^{(1/2-\epsilon/4)} h^{1/2-\epsilon/4}$ for some
$\newxconst>0$ independent of $h$.  A further manipulation using
$|\tau_n-\tau_k|\ge h$ then yields~\cref{eq:2x} with
$\gamma=1/2$. Thus, there are two steps to
verifying~\cref{eq:2x}: first, derive pathwise estimates of
certain stochastic integrals and, second, show that the time step
$\tau_{j+1}-\tau_j$ is not too small relative to $h$. We verify
\cref{ass:newx} for the {\EM} method with different time-stepping
methods in the proofs of \cref{thm:full} (fixed time-stepping)
and \cref{thm:ada,thm:adb} (adaptive bounded diffusions).

\subsection{Preliminary lemma}
Before giving the main convergence result in
\cref{thm:global_alt}, we give two lemmas required for its
proof. The following result plays the role normally assumed
by Gronwall's inequality in proving \cref{thm:global_alt}:
the local truncation error will determine $C_2$ and the
$J_{kn}$ is one of two terms that control the global error
$y_n-y(\tau_n;\tau_k,y_k)$. Then, \eqref{eq:ass_Jnm2}
describes how the local truncation error affects the global
error. We derive~\eqref{eq:7}, which shows $J_{kn}$ is
proportional to $L$ and hence $C_2$. That is, roughly, the
local truncation error $C_2$ controls the global error
$J_{kn}$. The result is adapted from~\citep{MR2387018}.
\begin{lemma}\label[lemma]{lemma:pullout}
  For a sequence $0=t_0<t_1<t_2<\cdots$, consider a set of vectors
  $\vec J_{kn}$ indexed by $k,n\in\naturals$. For constants
  $C_1,C_2,\gamma>0$, $\alpha>1$,
  suppose that %
\begin{equation}%
  \label[ineq]{eq:ass_Jnm2}
  \norm{\vec J_{k n}}%
  \le \norm{\vec J_{k\ell}}%
  \bp{1+ C_1|t_n-t_k|^\gamma}%
  + \norm{\vec J_{\ell+1,n}}
  + C_2|t_n-t_k|^{ \alpha}
\end{equation}
if $|t_k-t_{\ell}| \leq \frac{1}{2} |t_n-t_k|$ and
$|t_{\ell+1}-t_n| \leq \tfrac{1}{2} |t_n-t_k|$ and that \[ \vec
J_{kk}%
=\vec 0,\qquad%
\norm{\vec J_{k,k+1}}%
\le L |t_{k+1}-t_k|^{\alpha}\quad\text{for $L \coloneq \dfrac{2C_2}{ 1-2^{1-\alpha}}$.}
\]
Then 
\begin{equation}
  \label[ineq]{eq:7}
  \norm{\vec J_{k n}}%
  \le L |t_n-t_k|^{ \alpha}
  \quad \text{if  $|t_n-t_k|\le
    \delta$ }
\end{equation}
for  $\delta>0$ such that
$
2^{1-\alpha}C_1\delta^{\gamma}
=
{1-2^{1-\alpha}}$.
\end{lemma}%

\begin{proof} By the assumption on $\vec J_{kk}$ and $\vec
  J_{k,k+1}$, \cref{eq:7} holds for $|k-n|<2$. We complete the
  proof by induction on $n-k$ by showing \cref{eq:7} given
  \begin{equation}%
    \label[ineq]{eq:indhyp}%
    \text{ $\norm{\vec J_{k n'}} \leq L |t_{n'}-t_k|^{
        \alpha}$ \quad for all $k \leq n' < n$.}
  \end{equation}
  Let $\ell$ be the largest integer satisfying $k \leq \ell < n$
  and $|t_\ell-t_k| \leq \frac{1}{2}|t_n-t_k|$.  Then,
  also 
$|t_n-t_{\ell+1}| \leq \tfrac{1}{2} |t_n-t_k|$ and,
by~\cref{eq:indhyp},
  \[
  \norm{\vec J_{k\ell}}%
  \leq L|t_{\ell}-t_k|^{\alpha}%
  \le L 2^{-\alpha} |t_n-t_k|^{\alpha}, \qquad%
  \norm{\vec J_{\ell+1,n}}%
  \leq L|t_n-t_{\ell+1}|^{ \alpha}%
  \le L 2^{-\alpha} |t_n-t_k|^{\alpha}.
  \] 
From~\cref{eq:ass_Jnm2},
  \begin{align*}
    \norm{\vec J_{k n}}%
    &\le \norm{\vec J_{k\ell}}%
    \bp{1+C_1 |t_n-t_k|^{\gamma}}%
    + \norm{\vec J_{\ell+1,n}}%
    + C_2|t_n-t_k|^{ \alpha}\\
    &\le L2^{-\alpha}|t_n-t_k|^{ \alpha}
    \bp{1+C_1|t_n-t_{\ell}|^{\gamma}}%
    +L2^{-\alpha}|t_n-t_k|^{\alpha}+C_2 |t_n-t_k|^{\alpha}\\
    &\le \bp{L 2^{1-\alpha}%
      +L 2^{-\alpha}C_1|t_n-t_{\ell}|^{\gamma}+C_2}%
    |t_n-t_k|^{ \alpha}.
  \end{align*}
  Therefore, $ \norm{\vec J_{kn}}%
  \le L |t_n-t_k|^{ \alpha} $ provided that
  \[
  L 2^{1-\alpha}+ L 2^{-\alpha}C_1|t_n-t_{\ell}|^{\gamma} + C_2 
  \le L.
  \]
  By choice of $L$, this means
  \begin{equation}
    2^{-\alpha}{C_1|t_n-t_{\ell}|^{\gamma}}%
    \le {(1-2^{1-\alpha})-\frac{C_2}{L}}%
    =    \frac{1-2^{1-\alpha}}{2}.
    \label[ineq]{eq:15}
  \end{equation}
  Thus, the proof is complete as long as we choose $|t_n-t_\ell|$ to
  satisfy \cref{eq:15}. This is guaranteed if $|t_n-t_\ell|\le
  \delta$ as $\abs{t_n-t_\ell}\le \abs{t_n-t_k}\le \delta$.
\end{proof}

We now remove the assumption that $|t_n-t_k|\le \delta$ and allow
$|t_n-t_k|\le T$.
\begin{corollary}\label[corollary]{cor:this_is_it}%
  Suppose the conditions of \cref{lemma:pullout} hold. Then, if
  $0 \leq t_k \leq t_n\le T$,
  \[
  \norm{\vec J_{k n}} \le L'|t_n-t_k|^\alpha,\qquad
  L'=L\ceil*{\frac{ T }{\delta}}.
  \]
\end{corollary}
\begin{proof} If $\delta\ge T$, this is implied by
  \cref{lemma:pullout}. Assume $\delta<T$.  Divide the interval
  $[t_k,t_n]$ into the partition
  $t_k=t_{j_0}<t_{j_1}<\dots<t_{j_{k_{\max}}}= t_n$, where each
  $|t_{j_{k+1}}-t_{j_k}| \le \delta$ or $j_{k+1}-j_k=1$. Notice
  that, by choice of the partition, we can ensure that $k_{\max}\le
  \ceil{T/\delta}$. By \cref{lemma:pullout}, we must have
  $\norm{\vec J_{j_k,j_{k+1}}}\le L|t_{j_{k+1}}-
  t_{j_k}|^{\alpha}$. Finally, by~\cref{eq:ass_Jnm2} with
  $k=\ell=j_k$ and $\ell+1=j_{k+1}$ and $n=j_{k+2}$,
  \[
  \norm{\vec J_{j_{k},j_{k+2}}}%
  \le \norm{\vec J_{j_{k+1},j_{k+2}}}+C_2|t_{j_{k+2}}-t_{j_{k}}|^\alpha.%
  \]
  As $C_2\le L=2 C_2/(1-2^{1-\alpha})$,
  \[
  \norm{\vec J_{j_{k},j_{k+2}}}%
  \le 2L \,|t_{j_{k+2}}-t_{j_k}|^\alpha.
  \]
  Thus, the result now holds with $L'=2L$ for $|t_n-t_k|\le
  2\delta$. This argument can be repeated $k_{\max}$ times to
  gain $\norm{J_{kn}}\le L\ceil{T /\delta} |t_n-t_k|^\alpha$.
\end{proof}

\subsection{Main result}
We now use the two assumptions to show pathwise convergence.  

\begin{theorem}[convergence rate $\gamma$]
  \label[theorem]{thm:global_alt}
  Let \cref{ass:4.1,ass:4.2} hold for the SDE~\eqref{eq:sode} and
  \cref{ass:newx} hold for $\method^{\gamma}_{kn}$ with the
  partition $(\tau_0,\dots,\tau_N)\in \mathcal{T}$. Then, 
  \begin{equation}\label[ineq]{eq:star}
  \max_{n=0,\dots,N}\norm{\method^\gamma_{0n}(\vec y_0)%
    -\vec y(\tau_n; 0,\vec y_0)}%
  \le \newxconst\, \sdeconst(\epsilon,T)\, h^{\gamma-\epsilon},
  \end{equation}
  where
  \begin{align*}
  \sdeconst(\epsilon,T)%
   &\coloneq  T^{(1+\epsilon)/2}\\ &\quad+%
  \frac{\sdeconstd(\sdeconst+2) }{(1-2^{-\epsilon/6})}
  \pp{ T^{1+\epsilon/6}%
    +\frac{
      T^{2+\epsilon/6}(
\sdeconstd(\sdeconst+1))^{2/(1-2\epsilon)}}{(2^{\epsilon/6}-1)^{2/(1-2\epsilon)}}}.%
 \end{align*}
  
\end{theorem}
\begin{proof}
  Let $\vec y_n \coloneq \method_{0n}^{\gamma}(\vec y_0)$
  and consider $ \vec u_{k n}%
   \coloneq \vec y_n - \vec y(\tau_n; \tau_k, \vec y_k) +
  \vec X_{k n}(\vec y_0)$, for $X_{kn}(y_0)$ given in
  \cref{ass:newx}.  Note that $\vec u_{k,k+1}=\vec 0$ as
  \[\vec X_{k,k+1}=\vec \delta(\tau_k, \tau_{k+1}, \vec y_k)%
  =y(\tau_{k+1};\tau_k,y_k)-y_{k+1}\]
  by~\cref{eq:lte}. Further, $\vec u_{kn}$ captures the
  difference between the integrator $\vec S^\gamma_{kn}$ and
  the true solution corrected by $\vec X_{kn}$ and we expect
  it to be small. Rearranging, we have
  \begin{align}
    \vec y_n%
    &=\vec y(\tau_n; \tau_k, \vec y_k)%
    -\vec X_{k n}(\vec y_0) + \vec u_{k n}.\label{eq:b}
  \end{align}
  We prove the result by estimating the terms in \cref{eq:b} with
  $k=0$,
  \begin{equation}
  \norm{\vec y_n-\vec y(\tau_n; 0, \vec y_0)}%
  \le 
   \norm{\vec X_{0n}(\vec y_0)}%
  + \norm{\vec u_{0n}}.\label[ineq]{eq:smith}
  \end{equation}
  The first term satisfies $\norm{\vec X_{0n}(\vec y_0)}\le \newxconst
  T^{(1+\epsilon)/2} h^{\gamma-\epsilon}$ by \cref{ass:newx}, which
  is consistent with \cref{eq:star}. We will bound the second term by
  using~\Cref{cor:this_is_it} and hence show
  \cref{eq:star}. 

We will develop the inequality required
  to apply \cref{cor:this_is_it}.
  We first define two useful quantities $W_{k\ell}$ and
  $v_{k\ell n}$ and derive simple bounds on their
  magnitudes:

First  let 
\[z^1\coloneq \vec
  y(\tau_{\ell+1}; \tau_\ell,
  \vec y(\tau_\ell;\tau_k,y_k)-X_{k\ell}(y_0)+u_{k\ell}) -\vec
  X_{\ell,\ell+1}(\vec y_0)\quad\text{and}\quad z^2\coloneq \vec y(\tau_{\ell+1}; \tau_k, \vec
  y_k);
\]
then
define
  \begin{equation}
    \label{eq:2}
      \vec v_{k\ell n}%
       \coloneq \pp{\vec y\pp{\tau_n; \tau_{\ell+1},
        {z^1}}-z^1}%
      - \pp{\vec y\pp{\tau_n; \tau_{\ell+1}, {z^2}}
      -z^2}.
    \end{equation}
  This quantity is small when $|\tau_n-\tau_{l+1}|$ is
  small because, applying~\cref{eq:lte2},  
  we have
  \begin{align}
&    \norm{\vec v_{k\ell n}}\notag\\%
    &\le\sdeconstd \abs{\tau_n-\tau_{\ell+1}}^{1/2-\epsilon/3}\notag\\
    &\quad\times\norm{y\pp{\tau_{\ell+1}; \tau_\ell,
        y(\tau_\ell;\tau_k,y_k)%
        -X_{k\ell}(y_0)%
        +u_{k\ell}}-X_{\ell,\ell+1}(y_0)- y(\tau_{\ell+1};\tau_k,y_k)}\notag \\
    &=\sdeconstd \abs{\tau_n-\tau_{\ell+1}}^{1/2-\epsilon/3}%
    \notag\\%
    &\quad\times\Big\|y\pp{\tau_{\ell+1};%
        \tau_\ell,%
        y(\tau_\ell;\tau_k,y_k) -X_{k\ell}(y_0)+u_{k\ell}%
      }%
      -X_{\ell,\ell+1}(y_0)\notag \\%
  &\qquad    - y\pp{\tau_{\ell+1},\tau_\ell,y(\tau_\ell;\tau_k,y_k)}\Big\|%
    \notag\\%
    &\le\sdeconstd \abs{\tau_n-\tau_{\ell+1}}^{1/2-\epsilon/3}%
    \bp{\sdeconst \norm{X_{k\ell}(y_0)-u_{k\ell}}%
      + \norm{X_{\ell,\ell+1}(y_0)}}.  \label[ineq]{eq:vb}
  \end{align}
 Here, we use \cref{eq:boom} in the last step.  

Let
\begin{gather}\label{eq:W}
  \begin{split}
    W_{k\ell} &\coloneq y\pp{\tau_{\ell+1}; \tau_\ell,
      y(\tau_\ell;\tau_k,y_k)
        -X_{k\ell}(y_0)+u_{k\ell}}
    -y\pp{ \tau_{\ell+1};\tau_\ell, {y(\tau_\ell; \tau_k,y_k)}}\\%
&\quad    +X_{k\ell}(y_0)-u_{k\ell}. 
  \end{split}
\end{gather}
  As with $v_{k\ell n}$, this is bound by applying~\cref{eq:lte2},
  \begin{equation}\label[ineq]{eq:wb}
    \norm{W_{k\ell}}%
    \le \sdeconstd |\tau_{\ell+1}-\tau_\ell|^{1/2-\epsilon/3}%
    \norm{X_{k\ell}(y_0)-u_{k\ell}}.
  \end{equation}

  We now aim to apply \cref{cor:this_is_it} to $u_{kn}.$ We
  write $u_{kn}$ in terms of $v_{k\ell n}$ and $W_{k\ell}$,
  and use the above bounds to derive an inequality
  \eqref{eq:ass_Jnm2} for $u_{kn}$.  Start by applying~\cref{eq:b} three
  times, for $k\le \ell<n$,
  \begin{align*}
    \vec y_n%
    &= \vec y\pp{\tau_n; \tau_{\ell+1}, y_{\ell+1}}%
    - \vec X_{\ell+1,
      n}(\vec y_0)+ \vec u_{\ell+1, n}\notag\\
 &= \vec y\pp{\tau_n;
    \tau_{\ell+1}, y(\tau_{\ell+1};
    \tau_\ell,y_\ell)-X_{\ell,\ell+1}(y_0)} - \vec X_{\ell+1,
      n}(\vec y_0)+ \vec u_{\ell+1, n}\qquad\text{as $u_{\ell,\ell+1}=0$}\notag\\
    &= \vec y\pp{\tau_n; \tau_{\ell+1}, \smash{\underbrace{y\pp{\tau_{\ell+1};
    \tau_\ell,\vec y(\tau_\ell; \tau_k,\vec y_k) - \vec X_{k\ell}(\vec
    y_0)+\vec u_{k\ell}}-X_{\ell,\ell+1}(y_0)}_{=z^1}}}\rule[-5mm]{0mm}{1mm}%
\\&\quad    - \vec X_{\ell+1, n}(\vec y_0)+ \vec
    u_{\ell+1, n}.
  \end{align*}
  As $\vec y(\tau_n; \tau_\ell, \vec y(\tau_\ell; \tau_k,\vec
  y_k))=\vec y(\tau_n; \tau_k, \vec y_k)$, \cref{eq:2}
  gives
  \begin{align*}
    \vec v_{k\ell n}%
    &= \pp{\vec y_n +\vec X_{\ell+1, n}(\vec y_0)-\vec u_{\ell+1,
      n}}- \vec
    y(\vec \tau_n; \tau_k,\vec y_k)\\
    &\quad- \pp{\smash{\underbrace{y\pp{\tau_{\ell+1}; \tau_\ell,
    y(\tau_\ell,\tau_k,y_k)-X_{k\ell}(y_0)+u_{k\ell}}%
    -X_{\ell,\ell+1}(y_0)}_{=z^1}}}\vphantom{\underbrace{X}{X}} + y(\tau_{\ell+1},\tau_k,y_k).
  \end{align*}
  Now, substitute~\cref{eq:W},
  \begin{gather}
    \begin{split}
      v_{k\ell n} &= \pp{\vec y_n+\vec X_{\ell+1, n}(\vec y_0)-\vec
        u_{\ell+1, n}}%
 +X_{\ell,\ell+1}(y_0)      - \vec
      y(\vec \tau_n; \tau_k,\vec y_k)\\
      &\quad- \pp{W_{k\ell}-X_{k\ell}(y_0)+u_{k\ell}}.
    \end{split}\label{eq:love}
  \end{gather}
  Rearranging, and using $X_{k\ell}(y_0) +X_{\ell,\ell+1}(y_0)
  +X_{\ell+1,n}(y_0) =X_{k n}(y_0)$,
  \[
  \vec v_{k\ell n}+\vec u_{\ell+1, n} -\vec X_{k n}(\vec y_0)%
  +\vec u_{k\ell}+ W_{k\ell}%
  = \vec y_n - \vec y(\tau_n; \tau_k,\vec y_k).
  \]
  Starting from~\cref{eq:b},
  \begin{align}
    \vec u_{kn} &=\pp{\vec y_n - \vec y(\tau_n; \tau_k,\vec
      y_k)}%
    +\vec X_{kn}(\vec y_0)\notag\\
    &=\pp{ \vec v_{k\ell n}+\vec u_{\ell+1, n} -\vec X_{k
        n}(\vec y_0)+\vec u_{k\ell}+ W_{k\ell} }%
    +\vec X_{k n}(\vec y_0)\notag\\
    &=\vec v_{k\ell n}+\vec u_{\ell+1, n} +\vec
    u_{k\ell}+W_{k\ell}.\label{eq:clean}
  \end{align}
  Apply~\cref{eq:vb} and~\cref{eq:wb} to~\cref{eq:clean},
  \begin{align*}
    \norm{u_{kn}}&\le\sdeconstd
    \abs{\tau_n-\tau_{\ell+1}}^{1/2-\epsilon/3}%
    \bp{\sdeconst \norm{X_{k\ell}(y_0)-u_{k\ell}}%
      + \norm{X_{\ell,\ell+1}(y_0)}}\\%
    &\quad+\norm{u_{\ell+1,n}}+ \norm{u_{k\ell}}\\%
    &\quad+ \sdeconstd |\tau_{\ell+1}-\tau_\ell|^{1/2-\epsilon/3}%
    \norm{X_{k\ell}(y_0)-u_{k\ell}}.
  \end{align*}
  Then, with \cref{ass:newx},
  \begin{align*}
    \norm{u_{kn}}%
    &\le \norm{u_{k\ell}} \bp{%
      1+ \sdeconstd \pp{%
        \sdeconst |\tau_n-\tau_{\ell+1}|^{1/2-\epsilon/3}%
        + |\tau_{\ell+1}-\tau_\ell|^{1/2-\epsilon/3}}
    }\\
    &\quad+\norm{u_{\ell+1,n}}\\
&\quad+%
    \sdeconstd \newxconst |\tau_\ell-\tau_k|^{(1+\epsilon)/2}
    h^{\gamma-\epsilon} \bp{%
      \sdeconst |\tau_n-\tau_{\ell+1}|^{1/2-\epsilon/3} +
      |\tau_{\ell+1}-\tau_\ell|^{1/2-\epsilon/3}}\\ %
    &\quad+ \sdeconstd |\tau_n-\tau_{\ell+1}|^{1/2-\epsilon/3}
    \newxconst |\tau_\ell-\tau_{\ell+1}|^{\gamma+1/2-\epsilon/2}.
  \end{align*}
  Choose $\ell$ so that
  $|\tau_k-\tau_\ell|,|\tau_n-\tau_{\ell+1}|\le
  |\tau_n-\tau_k|/2$.  Then, we have (dropping all the
  $1/2^\alpha<1$ factors)
  \begin{align*}
    \norm{u_{kn}}%
    &\le\norm{u_{k\ell}}\bp{%
      1+ \sdeconstd \pp{%
        \sdeconst |\tau_n-\tau_{k}|^{1/2-\epsilon/3}%
        + |\tau_{n}-\tau_k|^{1/2-\epsilon/3}}%
    }\\
    &\quad+\norm{u_{\ell+1,n}}%
    \\%
&\quad+\sdeconstd \newxconst |\tau_n-\tau_k|^{(1+\epsilon)/2}
    h^{\gamma-\epsilon}%
    \bp{%
      \sdeconst |\tau_n-\tau_{k}|^{1/2-\epsilon/3} +
      |\tau_{n}-\tau_k|^{1/2-\epsilon/3}}\\ %
    &\quad+ \sdeconstd |\tau_n-\tau_{k}|^{1+\epsilon/6} \newxconst
    h^{\gamma-\epsilon}.
  \end{align*}
  Simplifying, we get
  \begin{align*}
    \norm{u_{kn}}%
    &\le\norm{u_{k\ell}}%
    \bp{ 1+ \sdeconstd (\sdeconst+1)
      |\tau_n-\tau_{k}|^{1/2-\epsilon} }%
    +\norm{u_{\ell+1,n}}\\%
    &\quad+\newxconst|\tau_n-\tau_k|^{1+\epsilon/6}
    h^{\gamma-\epsilon} \sdeconstd (\sdeconst+2).
  \end{align*}
Thus, we have shown that \cref{eq:ass_Jnm2} holds for
$J_{kn}=u_{kn}$ with
\[\gamma=1/2-\epsilon,\quad
  \alpha=1+\epsilon/6,\quad %
C_1=\sdeconstd(\sdeconst+1),
\] and
  \[
  C_2%
  = \newxconst h^{\gamma-\epsilon}\sdeconstd 
    (\sdeconst+2).
  \]
  \cref{cor:this_is_it} implies that
\[
\norm{\vec u_{k n}} \le L\ceil{ T /\delta}
  T^{1+\epsilon/6}\le L
  \pp{T^{1+\epsilon/6}+\frac{T^{2+\epsilon/6}}{\delta}},
\]
 for %
  $L=2 C_2/(1-2^{-\epsilon/6})$ and $\delta$ such that
  $
  2^{-\epsilon/6}C_1 \delta^{1/2-\epsilon}%
  ={1-2^{-\epsilon/6}}$.
Rearranging the equation for $\delta$, we have
\[
  \delta^{1/2-\epsilon}%
  =\frac{2^{\epsilon/6}-1}{C_1}%
  \quad\Rightarrow\quad
  \delta%
  =\frac{(2^{\epsilon/6}-1)^{2/(1-2\epsilon)}}{C_1^{2/(1-2\epsilon)}}.
  \]

  Returning to \cref{eq:smith}, we see that
\[
\norm{y_n-y(\tau_n;0,\vec y_0)}%
\le \newxconst T^{(1+\epsilon)/2} h^{\gamma-\epsilon}%
+ L\pp{T^{1+\epsilon/6}+\frac{ T^{2+\epsilon/6} }{\delta}}.
\]

Define 
  \[
  \sdeconst(\epsilon,T)%
   \coloneq T^{(1+\epsilon)/2}+%
  \frac{2 C_2/(\newxconst h^{\gamma-\epsilon})} {1-2^{-\epsilon/6}}
  \pp{ T^{1+\epsilon/6}%
    +\frac{
      T^{2+\epsilon/6}C_1^{2/(1-2\epsilon)}}{(2^{\epsilon/6}-1)^{2/(1-2\epsilon)}}}.%
  \]
Note that $\sdeconst(\epsilon,T)$ is independent of $h$ and
$\newxconst$ (by definition of $C_2$). We have
\[
\norm{y_n-y(\tau_n;0,\vec y_0)}%
\le  \newxconst \sdeconst(\epsilon,T) h^{\gamma-\epsilon}
\]
and \cref{eq:star} is
  now proved.
\end{proof}

\section{Fixed time-steps}%
\label{sec:em}%

To demonstrate the theory, we consider the case of fixed time-steps with the {\EM} method.  In this case, \cref{ass:bluenew}
holds with $\gamma=1/2$ with $\xconst$ uniform in the step size \citep[Corollary
10.17]{friz_multidimensional_2010}. This means \cref{thm:global1}
does not prove convergence in any sense and we need
Theorem~\ref{thm:global_alt} even to prove convergence, as well
as to establish the correct rate of convergence.  The
following  lemma is key to establishing the bound on
the truncation-error sum necessary for
Theorem~\ref{thm:global_alt}.
For $p\ge 1$, let $L^p(\Omega)$ denote the Banach space of
real-valued random variables $X$ with finite $p$th moments
and norm $\norm{X}_{L^p(\Omega)} \coloneq
(\mean{\abs{X}^p})^{1/p}$.
\begin{lemma}\label[lemma]{lemma:p}
  Consider predictable processes $\{\phi_N(s)\colon s\in[0,T]\}$
  for $N\in\naturals$. Fix $p\ge 1$ and suppose that there
  exists $K_p^*>0$ such that
  \begin{equation}\label[ineq]{eq:one_alt}
    \norm{\bar{\phi}_N^2}_{L^p(\Omega)}
    \le K_p^*,  \qquad\text{for }
    \bar{\phi}_N%
     \coloneq \sup_{0\le  s \le T} \abs{\phi_N(s)}.
  \end{equation}
  Choose $\lambda\ge 0$ and $i\in\{0,1,\dots,m\}$. Define
  \begin{align*}
    X^N(t)%
     \coloneq &\frac{1}{N^\lambda}%
    \int_0^t \phi_N(s)\,dW^i(s).
  \end{align*}
  Then, for all $\epsilon>0$, there exists $C\in
  L^p(\Omega)$ such that
  \begin{align}
    \norm{X^N(s)-X^N(t)} %
    \le C \frac1{N^{\lambda-\epsilon/2}} \times%
    \begin{cases}    |s-t|,&i=0,\\
      |s-t|^{1/2-\epsilon},&i=1,\dots,m,  
    \end{cases}\label[ineq]{eq:loon}
  \end{align}
  for $ N\ge 1$ and $ 0\le s<t\le T$. Further, we can choose $C$
  so that $\norm{C}_{L^p(\Omega)}\le \bar C(\epsilon,p,T,K_p^*)$  for
  some $\bar C(\epsilon,p,T,K_p^*)$ independent of $\phi_N$.
\end{lemma}
\begin{proof} 
  Note that $\phi_N(s)$ is predictable and $\int_r^t \phi_N(s)^2
  \,ds \le \bar{\phi}^2_N |t-r|$.  In the cases that $i\in
  \{1,\dots,m\}$,  \cref{prop:recent} (with
  $\xi=\bar \phi_N^2$ and $\mu=p/2$) shows that the modulus of continuity
  $\omega_Y$ of $Y \coloneq N^\lambda X^N$  satisfies %
  \[
  \omega_Y(h)%
  = \sup_{t,s\in [0,T],\; |t-s|\le h} \norm{ \vec Y(t)-\vec Y(s)  }%
  \le C_{p} h ^{1/2-\epsilon},\qquad 0< h\le T,
  \]
  for a constant $C_{p}\in L^p(\Omega)$ with
  $\norm{C_p}_{L^p(\Omega)}\le \bar{ C}_p$, where $\bar C_p$
  depends only on $\epsilon$, $p$, $T$, and $K_p^*$. Here
  the constant $C_p$ may depend on $N$ and the bound on
  $\omega_Y$ does not imply \cref{eq:loon}.  For $i=0$, $
  \norm{Y(t)-Y(s)}%
  \le |t-s| \bar \phi_N$, which means $\omega_Y(h)\le C_{p}
  h$ for $C_{p} \coloneq \bar \phi_N$.

  It remains to show that $C_p$ can be chosen uniformly in $N$.
  For the case $i=1,\dots,m$, let 
  \[
  A^N%
   \coloneq \sup_{0\le s,t\le T}%
  \frac{\norm{X^{N}(t)-X^{N}(s)}\,N^{\lambda-\epsilon/2}}%
  {|t-s|^{1/2-\epsilon}}.
  \]
  If $A^N \le C$, then \cref{eq:loon} holds for
  $i=1,\dots,m$.  As $X^N=N^{-\lambda} Y$, we see that
  $A^N=\sup_{0\le s,t\le T}%
  \norm{Y(t)-Y(s)}N^{-\epsilon/2}/|s-t|^{1/2-\epsilon}$ and
  $\mean{\abs{A^N}^p} \le \bar{C}_p^p N^{-\epsilon p/2}$.  For
  any $\delta>0$, Chebyshev's inequality gives
  \[
  \sum_{N=1}^\infty%
  \prob{A^N>\delta}%
  \le \sum_{N=1}^\infty \frac{\bar{C}_{p}^{p}   }%
  {N^{\epsilon p/2}\delta^{p}}.
  \]
  If $p>2/\epsilon$, the sum converges and the
  Borel--Cantelli lemma implies that $A^N\to 0$ almost
  surely. Let $C\coloneq \sup_{N\ge 1} A^N$ and note that
  \[
  \norm{C}_{L^p(\Omega)}%
  \le \pp{\sum_{N=1}^\infty \norm{A^N}^p_{L^p(\Omega)}}^{1/p}%
  \le\bar{C}_p \pp{\sum_{N=1}^\infty
  \frac{1}{N^{\epsilon p}}}^{1/p}<\infty.
  \]
  Thus, $C\in L^p(\Omega)$ for $p$ large and this extends to any
  $p\ge 1$ by Jensen's inequality. We have shown that
  \cref{eq:loon} holds for a constant $C$ independent of
  $N$. A similar argument applies for
  the case $i=0$.
\end{proof} %

We now prove convergence of the {\EM} method with fixed
time-steps. Similar results are given in
\citep{MR1625576,MR2320830} and our result gives more
details about the constants.
\begin{theorem}[fixed time-steps]%
  \label[theorem]{thm:full}%
  Let \cref{ass:4.1,ass:4.2} hold for the SDE~\eqref{eq:sode} and
  let $g_j\in \Cb^2(\real^d,\real^d)$ for $j=0,\dots,m$.  Let $y_n$
  be the fixed time-stepping {\EM} approximation~\eqref{eq:em} at
  uniformly spaced times $t_n=n h$ for $h=T/N $ and some $N>0$,
  with initial condition $y_0\in\real^d$.  Then, for all
  $0<\epsilon<1/2$, there exists a random variable $\fixedconst$
  such that, almost surely,
  \[
  \sup_{n=0,1\dots,N}
  \norm{y_n
    -\vec y(t_n; 0,\vec y_0)}%
  \le \sdeconst(\epsilon,T) \fixedconst
  \frac{1}{N^{1/2-\epsilon}},\qquad N\ge 1,
  \]
  and $\norm{\fixedconst}_{L^p(\Omega)}\le
  \max_j\norm{g_j}^2_{\Cb^2} K(\epsilon, p,T)$ for a constant
  $K(\epsilon,p,T)$ independent of $g_j$ and $N$.
\end{theorem}
\begin{proof}
  Fix
  $\epsilon\in(0,1/2)$. From~\cref{eq:lte_em}, \[\delta(t_{n},
  t_{n+1}, y_n)=\sum_{i,j}\int_{t_n}^{t_{n+1}} \int_{t_n}^s
  q_{ij}(y(r;t_n,y_n))\,dW^i(r)\,dW^j(s),\] where $q_{ij}$
  is defined in~\cref{eq:q} and
  \begin{equation}\label{eq:group1}
  \vec X_{kn}(y_0)%
  =\sum_{i,j=0}^m\int_{t_k}^{t_n} %
  \int_{\shat}^s
  q_{ij}(y(r;\rhat,\yhat(r))) \,dW^i(r)
  \,dW^j(s),\qquad%
  \text{$\yhat(r)=y_k$ if $\rhat=t_k$}.
  \end{equation}
  We now establish \cref{ass:newx} in order to apply
  \cref{thm:global_alt}. Thus, we seek $\newxconst$ and $\gamma$ such
  that
  \begin{equation}\label[ineq]{eq:group}
    \norm{\vec X_{kn}(\vec y_0)}%
    \le \newxconst (t_n-t_k)^{(1+\epsilon)/2} h^{\gamma-\epsilon}.
  \end{equation}
  Let $\bar g\coloneq \max_j \norm{g_j}_{\Cb^2}$ and
  \[
  \psi_{ij}(s)%
   \coloneq \frac{1}{\bar g^2}\int_{0}^s q_{ij}\pp{r;\rhat,\yhat(r)}\,dW^i(r),\qquad 0\le s\le T,
  \]
  so that $X_{kn}(y_0)=\bar g^2\sum_{i,j=0}^m
  \int_{t_k}^{t_n}
  \pp{\psi_{ij}(s)-\psi_{ij}(\shat)}\,dW^j(s)$.  Now
  $\psi_{ij}(s)$ is a predictable process and
  $\norm{q_{ij}}$ is bounded by $\bar g^2$. Then
  \cref{lemma:p} applies with $\epsilon\mapsto \epsilon/12$,
  $\lambda=0$, and $X^N=\psi_{ij}$, so that there exists
  $C^1>0$ such that
  \begin{align}\label[ineq]{star}
  \sup_{0\le s \le T}%
  \norm{ \psi_{ij}(s)-\psi_{ij}(\shat)}%
  \le C^1 N^{-\epsilon/24}\times
  \begin{cases}
    |s-\shat|,& i=0,    \\[0.9em]
    |s-\shat|^{1/2-\epsilon/12},& i=1,\dots,m.
  \end{cases}
  \end{align}
 Further, $C^1\in L^p(\Omega)$ and $\norm{C^1}_{L^p(\Omega)}\le \bar
  C^1$ for a constant $\bar C^1$ independent of $g_j$.
  
  As $|s-\shat| \le h=T/N$, \cref{star} implies that
  \cref{eq:one_alt} holds with $\phi_N(s)=
  (\psi(s)-\psi(\shat))N^{\lambda}$ and
  $\lambda=1/2-\epsilon/8$ for $i=1,\dots,m$, and also for
  $i=0$ as $1-\epsilon/24>1/2-\epsilon/8$ (as
  $\epsilon\in(0,1/2)$).  As $\shat<s$, $\phi_N(s)$ is again a
  predictable process. Applying \cref{lemma:p} once more
  with $\epsilon\mapsto \epsilon/4$, we find a $ C^2\in
  L^p(\Omega)$ such that for $N\in\naturals$
  \[
  \norm{
  \int_{t_k}^{t_n} \pp{\psi_{ij}(s)-\psi_{ij}(\shat)} \,dW^j(s)}
  \le C^2\frac{1}{N^{\lambda-\epsilon/8}}\times
  \begin{cases}
      |t_n-t_k|,& j=0,\\[0.5em]
      |t_n-t_k|^{1/2-\epsilon/4}, & j=1,\dots,m.\\
  \end{cases}
  \]
  Then, summing over $i,j=0,\dots,m$, we find $\newxconst\in L^p(\Omega)$ such that
  \[
  \norm{\vec X_{kn}(\vec y_0)}%
  \le   \newxconst h^{1/2-\epsilon/4} |t_n-t_k|^{1/2-\epsilon/4}.
  \]
  Further $\newxconst=\bar g^2 K(\epsilon,p)$ for a random
  variable $K(\epsilon,p)$ such that
  $\norm{K(\epsilon,p)}_{L^p(\Omega)}\le \Kbar(\epsilon,p,T)$ for
  a $\Kbar (\epsilon,p,T)$ independent of $g_j$.

  Finally, $X_{nn}=0$ and $h\le |t_n-t_k|$ for $n\ne k$, so that
  \[
  \norm{\vec X_{kn}(\vec y_0)}%
  \le \newxconst h^{1/2-\epsilon} |t_n-t_k|^{(1+\epsilon)/2}.
  \]
  This gives the required bound in~\cref{eq:group} with
  $\gamma=1/2$.  Thus, we have found a constant $\newxconst\in
  L^p(\Omega)$ such that~\cref{eq:group} holds.
  \cref{thm:global_alt} now applies to complete the proof.
\end{proof}

\section{Adaptive time-stepping with bounded diffusions}
\label{sec:bd}
To demonstrate the theory for random times, we introduce an
adaptive time-stepping method based on the method of bounded
diffusions~\citep{MR1722281}.  For {\EM}, the local truncation error
\[
\delta(\tau_n, \tau_{n+1}, y_n)%
=\sum_{i,j=0}^m \int_{\tau_n}^{\tau_{n+1}}%
\int_{\tau_n}^s%
q_{ij}(y(r;\tau_n,y_n)) \, dW^i(r) \,dW^j(s)
\]
and we control this error by selecting the time step
$\tau_{n+1}-\tau_n$ as follows. First, fix $\alpha>0$ as a
parameter and denote the maximum time-step by $h=T/N$ for a
discretisation parameter $N\in \naturals$. Suppose that $y_n$ is
a given approximation at a stopping time $\tau_n$.  We consider
two schemes for choosing $\tau_{n+1}$.
\begin{description}
\item[{Adaptive-I}] Choose $\tau_{n+1}$ to be the largest $t\in
  [\tau_n,\tau_n+h]\cap [0,T]$ such that for $i=1,\dots,m$
  \begin{align}
    \max_{j=1,\dots,m}\norm{q_{ij}(y_n)}^{1/2}%
    \,\abs{W^i(t)-W^i(\tau_n)}%
    &\le \alpha h^{1/2}.\label[ineq]{eq:bdc}
  \end{align}
  
\item[{Adaptive-II}] Choose $\tau_{n+1}$ to be the largest $t\in
  [\tau_n,\tau_n+h]\cap [0,T]$ such that  \begin{equation}\label[ineq]{eq:bdalt}
    \max_{i,j=1,\dots,m} \norm{q_{ij}(y_n) }%
    \abs{ \int_{\tau_n}^{\tau_{n+1}} \int_{\tau_n}^s
      dW^i(r)\,dW^j(s)} %
    \le \frac12\alpha^2 h.
  \end{equation}
\end{description}
Notice that $\tau_{n+1}$ is also a stopping time.  We define the
next approximation $y_{n+1}$ at time $\tau_{n+1}$
by~\cref{eq:em}. Given $y_0$ and $\tau_0=0$, this rule defines an
approximation $y_n$ at stopping times $\tau_n$ for all
$n=0,\dots,N$ where $\tau_N= T$. That is,
$(\tau_0,\dots,\tau_N)\in \mathcal{T}_{\operatorname{stop}}$.
The first method of choosing the time step $\tau_{n+1}-\tau_n$ is
equivalent to finding the first exit time $t$ of
$(t,W^1(t),\dots,W^m(t))$ from a cuboid
$[0,a_0]\times[-a_1,a_1]\times\dots\times[-a_m,a_m]$ with
\[
a_0%
\coloneq\min\Bp{h, T-\tau_n}, \qquad%
a_i%
\coloneq \min \Bp{h^{1/2}\frac{\alpha}{
    {\norm{q_{ij}(y_n)}^{1/2}}}\colon j=1,\dots,m},\quad i=1,\dots,m.
\]
\citet{MR1722281} give an algorithm for sampling a time step from
this distribution, which is used for the experiments in
\cref{ss:exps}.

For Adaptive-II, we replace~\cref{eq:bdc} with a term involving a
double stochastic integral to get~\cref{eq:bdalt}.  In the case
of diagonal noise,~\cref{eq:bdalt} simplifies to
\begin{equation}\label[ineq]{eq:rob}
     \norm{q_{jj}(y_n) }%
\abs{
  \pp{W^j(\tau_{n+1})-W^j(\tau_n)}^2-\pp{\tau_{n+1}-\tau_n}}
\le \alpha^2 h.
\end{equation}
In the case $\alpha^2>\norm{q_{jj}}$, this condition holds
automatically if~\cref{eq:bdc} holds and allows for longer time-steps to be taken.  See \cref{fig:1}.
 In general, it is
not clear how to sample such a time step and an approximate
method is utilised in \cref{ss:exps} for an example in  $m=1$ dimensions.

\begin{figure}
  \centering
  \includegraphics{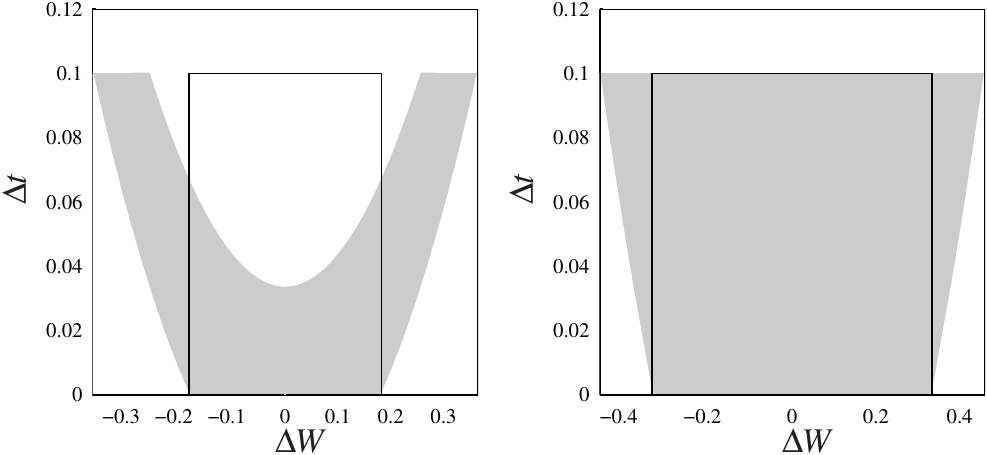} 
  \caption{The plots show the regions in $(\Delta t,\Delta
    W)=(\tau_{n+1}-\tau_n, W^j(\tau_{n+1})-W^j(\tau_n))$ where
    $\tau_{n+1}-\tau_n\in[0,h]$ and \cref{eq:bdc} holds
    (rectangular box) and \cref{eq:rob} holds (grey region)
    hold for $\alpha=1$, $h=0.1$ and (left)
    $\norm{q_{jj}(y_n)}=3$ where $\alpha^2<\norm{q_{jj}(y_n)}$
    and (right) $\norm{q_{jj}(y_n)}=0.9$ where
    $\alpha^2>\norm{q_{jj}(y_n)}$.}\label{fig:1}
\end{figure}

\subsection{Pathwise convergence for Adaptive-I}
There are two parts to the proof of convergence: first, assuming
smoothness of $g_j$, we establish that the time steps are not too
small in \cref{lemma:control} and then, in \cref{thm:ada}, we show
the conditions of \cref{thm:global_alt} hold.

\begin{lemma}%
  \label[lemma]{lemma:control}%
  Suppose that $g_j\in \Cb^2(\real^d,\real^d)$ for
  $j=0,\dots,m$. Choose $(\tau_0,\dots,\tau_N)\in {\cal T}$ such
  that~\eqref{eq:bdc} holds for parameters $\alpha,h>0$.  For all
  $0<\delta< 1$, there exists a random variable $C_\delta>0$
  (independent of $g_j$ and $y_0$ and $\alpha$) such that
  \begin{equation}\label[ineq]{eq:hitme2}
  h%
  \le \KBI  |\tau_{n+1}-\tau_n|^{1-\delta}\qquad\text{for
    $n=0,\dots,N-2$,}
  \end{equation}
  where $ \KBI  \coloneq  
  T^{\delta}+ {C_\delta^2 \max_j\norm{g_{j}}_{\Cb^2}^2} /
  {\alpha^2}$.
\end{lemma}
\begin{proof} Choose  a random variable
  $C_\delta$ such that $|W^i(t)-W^i(s)| \le C_\delta
  |t-s|^{(1-\delta)/2}$ for $i=1,\dots,m$ and $0\le s,t\le T$.

  If $\tau_{n+1}-\tau_n=h$, we have $h\le T^\delta
  |\tau_{n+1}-\tau_n|^{1-\delta}$ and \cref{eq:hitme2}
  holds. Alternatively, $\tau_{n+1}-\tau_n<h$ and, as $n\le N-2$,
  we know that $\tau_{n+1}<T$. In this case, we must have that
  $\tau_{n+1}$ satisfies~\cref{eq:bdc} with equality at
  some specific $i,j=1,\dots,m$. That is,
  \[
  \norm{q_{ij}(y_n)}^{1/2}%
  \abs{    {W^i(\tau_{n+1})-W^i(\tau_n)} }%
  =\alpha h^{1/2}.
  \]
  Then,
  \[
  C_\delta |\tau_{n+1}-\tau_n|^{(1-\delta)/2}%
  \ge \abs{W^i(\tau_{n+1})-W^i(\tau_n)} %
  = \frac{\alpha h^{1/2}}{\norm{q_{ij}(y_n)}^{1/2} }.
  \]
  In other words,
  \[
  h \le { \frac{C_\delta^2
      \norm{q_{ij}(y_n)}}{\alpha^2}}%
  |\tau_{n+1}-\tau_n|^{1-\delta}%
  \le \KBI|\tau_{n+1}-\tau_n|^{1-\delta}
  \]
for the given $\KBI$, as required.
\end{proof}

The next theorem describes an error bound for Adaptive-I. To
leading order, the constant $K_{\mathrm A}$ in the error bound
scales like $\bar g^{1+\epsilon}$ for $\bar g=\max_j
\norm{g_j}_{\Cb^2}$, whilst the corresponding constant
$K_{\mathrm F}$ in \cref{thm:full} scales like $\bar g^2$. This
is a  more favourable scaling of the error estimate and says
the error bound scales nearly linearly with the magnitude of the
vector fields $g_j$.

\begin{theorem}[adaptive-I]%
  \label[theorem]{thm:ada}
  Let \cref{ass:4.1,ass:4.2} hold for the SDE~\eqref{eq:sode} and
  suppose that $g_j\in \Cb^3(\real^d,\real^d)$.  Let $y_n$ denote
  the {\EM} approximation~\eqref{eq:em} at times $\tau_n$ given by
  Adaptive-I with $h=T/N$, some $N\in\naturals$. Then, for
  $0<\epsilon<1/2$ and $y_0\in\real^d$, there exists a random
  variable $\adapconst>0$ such that, almost surely,
  \[
  \sup_{0\le \tau_n\le T}
  \norm{y_n%
    -\vec y(\tau_n; 0,\vec y_0)}%
  \le \sdeconst(\epsilon,T)\, \adapconst 
  \frac{1}{N^{1/2-\epsilon}},\qquad%
  N\ge 1,
  \]
  where $\norm{\adapconst}_{L^p(\Omega)}\le K_1(\epsilon, p,
  \alpha,T) \max _j\norm{g_j}_{\Cb^2}^{1+2\epsilon}+K_2(\epsilon,
  p, \norm{g_j}_{\Cb^3}, \alpha, T)/N^{1/2}$ for some
  $K_1(\epsilon,p,\alpha,T)$ independent of $g_j$ and $N$ and
  $K_2(\epsilon, p, \norm{g_j}_{\Cb^3}, \alpha, T)$ independent
  of $N$.
\end{theorem}

\begin{proof} 
 We have as in \cref{eq:group1}
 \begin{equation*}
   \vec X_{kn}(y_0)%
   =\sum_{i,j=0}^m\int_{t_k}^{t_n} %
   \int_{\shat}^s
   q_{ij}\pp{y(r;\rhat,\yhat(r))} \,dW^i(r)
   \,dW^j(s),\qquad %
   \text{$\yhat(r)=y_k$ if $\rhat=\tau_k$}.
  \end{equation*}
  Expanding the terms $i,j=1,\dots,m$ again using \Ito's
  formula
  \begin{gather}
    \begin{split}
    \vec X_{kn}(y_0)%
    &=\sum_{i,j=1}^m\int_{t_k}^{t_n} %
    \int_{\shat}^s q_{ij}\pp{y(\rhat;\rhat,\yhat(r))}
    \,dW^i(r)
    \,dW^j(s)\\
    &\quad+\sum_{j=0}^m\int_{t_k}^{t_n} %
    \int_{\shat}^s
    q_{0j}\pp{y( r;\rhat,\yhat(r))} \,dr \,dW^j(s)\\
    &\quad+\sum_{i=1}^m\int_{t_k}^{t_n} %
    \int_{\shat}^s q_{i0}\pp{y( r;\rhat,\yhat(r))}
    \,dW^i(r)
    \,ds\\
    &\quad+ \sum_{i,j=1,\ell=0}^m\int_{t_k}^{t_n} %
    \int_{\shat}^s \int_{\shat}^r q_{\ell ij}\pp{y(
      u;\uhat,\yhat(u))}\,dW^\ell(u) \,dW^i(r) \,dW^j(s),  
    \end{split}\label{eq:quit}
  \end{gather}
  where $q_{\ell i j}\coloneq\mathcal{L}^\ell q_{ij}$ for $\mathcal{L}^0
  \phi\coloneq D \phi g_0+\frac 12 \sum_{j=1}^m D^2 \phi (g_j,g_j)$ and
  $\mathcal{L}^\ell\phi\coloneq D \phi g_\ell$. 
  Fix $\epsilon\in(0,1/2)$. Following the same arguments as
  in the proof of \cref{thm:full}, the last three terms here
  are bounded by $\newxconst^2 h^{1-\epsilon}
  |t_n-t_k|^{(1+\epsilon)/2}$ for a constant $\newxconst^2$
  that satisfies $\norm{\newxconst^2}_{L^p(\Omega)}\le
  K\pp{\epsilon,\alpha,p,\norm{g_j}_{\Cb^3}}$ for some
  $K\pp{\epsilon,\alpha,p,\norm{g_j}_{\Cb^3}}$ independent
  of $N$.
  
  To bound the first term, we use a more refined argument. Let
  $\bar g \coloneq \max_j \norm{g_j}_{\Cb^2}$ and
  \[
  \phi_N(s)%
   \coloneq \frac1{\bar g h^{1/2}}%
  \int_{\shat}^s q_{ij}\pp{y(\rhat; \rhat,
    \yhat(r))}\,dW^i(r)%
   = \frac1{\bar g h^{1/2}}%
  \int_{\shat}^s q_{ij}\pp{ \yhat(r)}\,dW^i(r)
  \]
  for $\tau_k=\rhat\le r<\tau_{k+1}$.  Here $\phi_N(s)$ is
  continuous and adapted (as $\shat$ is a stopping time and
  $\shat<s$) and hence predictable.  By~\cref{eq:bdc},
  \[
  \norm{q_{ij}\pp{y(\rhat; \rhat, \yhat(r))} }^{1/2}\,\abs{W^i(\tau_{n+1})-W^i(\tau_n)} %
  \le \alpha h^{1/2},\qquad j=1,\dots,m.
  \]
As $\norm{q_{ij}}\le \bar g^2$, this gives 
  \[
  \sup_{0\le s \le T} \norm{ \phi_N(s)}^2%
  = \frac{1}{\bar g^2 h} \sup_{0\le s\le T}%
  \norm{\int_{\shat}^s%
    {q_{ij}\pp{y(\rhat; \rhat, \yhat(r))} }\,dW^i(r)}^2 %
  \le \alpha^2.
  \]
  Hence, \cref{eq:one_alt} holds and \cref{lemma:p} applies
  with $\lambda=1/2$ and \[X^N(t)=h^{1/2}\int_0^t
  \phi_N(s)\,dW^i(s).\] This gives a $c_{ij}\in
  L^p(\Omega)$ such that
  \[
  \frac{1}{\bar g}%
  \norm{ \int_{\tau_k}^{\tau_n}%
    \int_{\shat}^{s}%
    q_{ij}\pp{y(\rhat;\rhat, \yhat(r))}
    \,dW^i(r)\,dW^j(s)}%
  \le c_{ij} \,{|\tau_n-\tau_k|^{1/2-\epsilon/4}}\,{h^{1/2-\epsilon/8}}.
  \]
  We may choose $c_{ij}$ so that its $L^p(\Omega)$ norm is
  independent of $h$ and $\bar g$.  This  provides the
  necessary bound on the first term of \cref{eq:quit}. 

  Taken all the terms in \cref{eq:quit} together, we can
  find $\newxconst^1$ such that
  \[
  X_{kn}(y_0)%
  \le \newxconst^1 {|\tau_n-\tau_k|^{1/2-\epsilon/4}}{h^{1/2-\epsilon/8}}%
  + \newxconst^2 {}{h^{1-\epsilon}} |\tau_n-\tau_k|^{(1+\epsilon)/2}.
  \]
  Here $\norm{\newxconst^1}_{L^p(\Omega)}\le \bar g
  K(\epsilon,p,\alpha,T)$ for some $K(\epsilon,p,\alpha,T)$
  independent of $N$ and $g_j$.

  Using Lemma~\ref{lemma:control} with $\delta=1/7$, we have
  the lower bound $h \le \KBI|\tau_n-\tau_k|^{1-\delta}=\KBI
  \abs{\tau_n-\tau_k}^{6/7}$ on the time step for
  $k,n=0,\dots,N-1$. Hence,
  \begin{align*}
    |\tau_n-\tau_k|^{1/2-\epsilon/4}%
    h^{1/2-\epsilon/8}%
    &\le |\tau_n-\tau_k|^{1/2-\epsilon/4} {h^{1/2-\epsilon}}
    {h^{7\epsilon/8}}\\
    &\le |\tau_n-\tau_k|^{1/2-\epsilon/4} {h^{1/2-\epsilon}}
    \,\KBI^{7\epsilon/8}\,
    |\tau_n-\tau_k|^{\epsilon 3 /4}\\
    &= |\tau_n-\tau_k|^{1/2+\epsilon/2}%
    \,{h^{1/2-\epsilon}}\KBI^{7\epsilon/8}. %
  \end{align*}
  Consequently,
  \[
  \norm{X_{kn}(y_0)}%
  \le \newxconst^1 \KBI^\epsilon 
  {|\tau_n-\tau_k|^{(1+\epsilon)/2}} {h^{1/2-\epsilon}} +
  \newxconst^2 |\tau_n-\tau_k|^{(1+\epsilon)/2} h^{1-\epsilon}.
  \] 
  We have all the conditions of \cref{thm:global_alt}, which
  gives the desired error bound for $n=0,\dots,N-1$. The final
  step from $\tau_{N-1}$ to $\tau_N$ may equal $|T-\tau_{N-1}|$,
  which may be very small and does not yield to
  \cref{lemma:control}. However, the error resulting from this
  step can be added into the $\newxconst^2$ term by utilising the
  bound in \cref{eq:3}.
\end{proof}

\subsection{Analysis of Adaptive-II}
The convergence result (\cref{thm:adb} below) for {\EM} with
Adaptive-II is similar to \cref{thm:ada} and the new time-stepping strategy behaves like Adaptive-I. However, the
method of proof is different and we will make use of Azuma's
inequality. First, we show the time step does not become too
small, similarly to \cref{lemma:control}.
\begin{lemma}%
  \label[lemma]{lemma:control2}%
  Suppose that $g_j\in \Cb^2(\real^d,\real^d)$ for
  $j=0,\dots,m$. Choose $(\tau_0,\dots,\tau_{N})\in\mathcal{T}$
  such that \cref{eq:bdalt} holds for parameters $\alpha,h>0$.
  For $0<\delta< 1$, there exists a random variable $C_\delta>0$
  (independent of $\alpha$, $h$, $g_j$, and $y_0$) such that
  \begin{equation}%
    \label[ineq]{eq:hitme3}%
    h%
    \le \KBI |\tau_{k+1}-\tau_k|^{1-\delta},\quad%
    \text{for $k=0,\dots,N-2$,}
  \end{equation}
  where $ \KBI=T^\delta+{ {2 C_\delta
      \max_j\norm{g_{j}}_{\Cb^2}^2 / \alpha^2}}$.
\end{lemma}
\begin{proof}
  Choose a random variable $C_\delta$ such that
  $\abs{\int_s^t W^i(r)\,dW^j(r)}%
  \le C_\delta |t-s|^{1-\delta}$ for $i,j=1,\dots,m$ and
  $0\le s,t\le T$.

  If $\tau_{k+1}-\tau_k=h$, then $h\le T^\delta
  |\tau_{k+1}-\tau_k|^{1-\delta}$ and \cref{eq:hitme3}
  holds. Otherwise, we must have $\tau_{k+1}-\tau_k<h$ and
  $\tau_{k+1}<T$. Then $\tau_{k+1}$ satisfies~\cref{eq:bdalt} with
  equality for some $i,j=1,\dots,m$:
  \[
  \norm{q_{ij}(y_k)}\,%
  \abs{\int_{\tau_k}^{\tau_{k+1}} \int_{\tau_k}^
    s\,dW^i(r)\,dW^j(s)} %
  =\frac 12\alpha^2 h.
  \]
  Then,
  \[
  C_\delta |\tau_{k+1}-\tau_k|^{1-\delta}
  \ge \abs{\int_{\tau_k}^{\tau_{k+1}} \int_{\tau_k}^s\,W^i(r)\,dW^j(s)}
  \ge \frac{\alpha^2 h}{2\norm{q_{ij}(y_k)} }.
  \]
  In other words,
  \[
  h \le { \frac{2 C_\delta \norm{q_{ij}(y_k)}}{\alpha^2}}%
  |\tau_{k+1}-\tau_k|^{1-\delta}%
  \le \KBI|\tau_{k+1}-\tau_k|^{1-\delta}.
  \]
  Thus,  \cref{eq:hitme3} holds and the proof is
  complete.
\end{proof}

The error bound for Adaptive-II found in the next theorem scales
(in terms of $g_j$)  like the one for Adaptive-I.
\begin{theorem}[adaptive-II]%
  \label[theorem]{thm:adb}
  Let \cref{ass:4.1,ass:4.2} hold for the SDE~\eqref{eq:sode} and
  suppose that $g_j\in \Cb^3(\real^d,\real^d)$.  Let $y_n$ denote
  the {\EM} approximation at times $\tau_n$ with adaptive increments
  given by \cref{eq:bdalt,eq:em}. Then, for
  $0<\epsilon<1/2$ and  $y_0\in\real^d$, there exists a random
  variable $\adapconst>0$ such that, almost surely,
  \[
  \sup_{0\le \tau_n\le T}
  \norm{y_n%
    -\vec y(\tau_n; 0,\vec y_0)}%
  \le \sdeconst(\epsilon,T)\, \adapconst 
  \frac{1}{N^{1/2-\epsilon}},\qquad%
  N\ge 1,
  \]
  where $\norm{\adapconst}_{L^p(\Omega)}\le K_1(\epsilon, p,
  \alpha,T)\norm{g_j}_{\Cb^2}^{1+\epsilon}
  +K_2(\epsilon,p,\alpha,\norm{g_j}_{\Cb^3},T)/N^{1/2}$ for some
  $K_1(\epsilon, p, \alpha,T)$ independent of $N$ and $g_j$, and
  some $K_2(\epsilon,p,\alpha,\norm{g_j}_{\Cb^3},T)$ independent
  of $N$.
\end{theorem}

\begin{proof}
  Following the proof of \cref{thm:ada}, it is enough to treat
  the term
  \[
  S_{kn}%
   \coloneq  \sum_{j=k}^{n-1} S_j,\qquad \text {for
  }S_k \coloneq \sum_{i,j=1}^m\int_{\tau_k}^{\tau_{k+1}} %
  \int_{\tau_k}^s q_{ij}(y_k) \,dW^i(r) \,dW^j(s) \text{ and $0\le
    k<n\le N$,}
  \]
  and show it satisfies the condition on $X_{kn}$ in
  \cref{ass:newx}.  By the optional stopping theorem, each
  $S_k$ has mean zero and, by~\cref{eq:bdalt}, $\norm{S_k}
  \le m^2 \alpha^2 h/2$ for $h=T/N$. Then, $S_{kn}$ is a sum
  independent random variables, each with mean zero and
  bounded by $m^2 \alpha^2 h/2$ .  Azuma's inequality (see
  \cref{azuma} with $\lambda=(n-k)^{1/2} h^{1-\epsilon/4}$)
  gives
  \[ 
  \prob{ \norm{S_{kn}} \ge \lambda}%
  \le 2
  \exp\pp{\frac{-4\lambda^2}{2(n-k) m^4 \alpha^4 h^2}}%
  =2\exp\pp{\frac{-2}{m^4 \alpha^4 h^{\epsilon/2}}}.
  \]
  Let $F_N\coloneq \{\omega\in\Omega\colon \sup_{0\le k<n\le N}
  \norm{S_{kn}(\omega)} /(n-k)^{1/2}\ge h^{1-\epsilon/4}\}$.
  Using $h=T/N$, 
  \[ 
  \prob{F_N} \le \sum_{0\le k<n\le N} \prob{ \norm{S_{kn}} \ge %
    \lambda} %
  \le  (N+1)(N+2) \exp\pp{\frac{-2N^{\epsilon/2} }{m^4 \alpha^4
      T^{\epsilon/2}}}.%
\]
Choose $C_\epsilon$ so that $( N+1)(N+2) \exp(-N^{\epsilon/2}/m^4
\alpha^4 T^{\epsilon/2}) \le C_\epsilon$ and
\[
  \prob{F_N}
  \le C_\epsilon \exp\pp{\frac{-N^{\epsilon/2}}{ m^4 \alpha^4
      T^{\epsilon/2}}}\qquad\text{for $N\in \naturals$ .}
  \]
  Then $\sum_{N=1}^\infty \prob{F_N}<\infty$ and the
  Borel--Cantelli lemma applies, to give
  \begin{equation}
    \label[ineq]{eq:1}
    \sup_{0\le k<n\le N} \frac{\norm{S_{kn}}
      N^{1-\epsilon/4}}{(n-k)^{1/2}}%
    \le C, 
  \end{equation}
almost surely,  for some random variable $C$.

  By \cref{lemma:control2}, for each $\delta>0$, there is a
  $\KBI$ such that
  \[
  \pp{\frac{1}{\KBI} \frac{T}{N}}^{1/(1-\delta)} \le
  |\tau_{k+1}-\tau_k|%
  \]
  and, summing  $|\tau_{k+1}-\tau_k|,\dots,|\tau_{n}-\tau_{n-1}|$,
  \[
  (n-k)\pp{\frac{1}{\KBI} \frac{T}{N}}^{1/(1-\delta)} \le
  |\tau_n-\tau_k|.%
  \]
  Choose $\delta$ so that $(1+\epsilon)/2(1-\delta)=1/2+3\epsilon/4$. Then,
  \[
  \frac{N^{1/2-\epsilon}}{ |\tau_n-\tau_k|^{(1+\epsilon)/2}}%
  \le \pp{\frac{\KBI}{ T}}^{(1+\epsilon)/2(1-\delta)}
  \frac{N^{1/2 -\epsilon+(1+\epsilon)/2(1-\delta)}}{|n-k|^{(1+\epsilon)/2}}%
  \le \pp{\frac{\KBI}{ T}}^{1/2+3\epsilon/4}
  \frac{N^{1-\epsilon/4}}{|n-k|^{1/2}}.%
  \]
  Now use
  \cref{eq:1}  to gain
  \[
  \sup_{0\le k<n\le N}%
  \frac{\norm{S_{kn}} N^{1/2-\epsilon}}{|\tau_n-\tau_k|^{(1+\epsilon)/2}} \le
  \pp{\frac{\KBI}{ T}
  }^{1/2+3\epsilon/4}\sup_{0\le k<n\le N}%
  \frac{\norm{S_{kn}} N^{1-\epsilon/4}}{(n-k)^{1/2}}%
  \le C\pp{\frac{\KBI}{ T}
  }^{1/2+3\epsilon/4}.
  \]
  Thus, $\norm{S_{kn}} \le \tilde C |\tau_n-\tau_k|^{(1+\epsilon)/2}
  h^{1/2-\epsilon}$ for a constant $\tilde C$. The remainder
  of the argument is the same as in the proof of
  \cref{thm:ada}. In this case, the dependence on
  $\norm{g_j}_{C^2_b}$ comes from $\KBI$.
\end{proof}                     %

\begin{figure}
 \centering
  \includegraphics{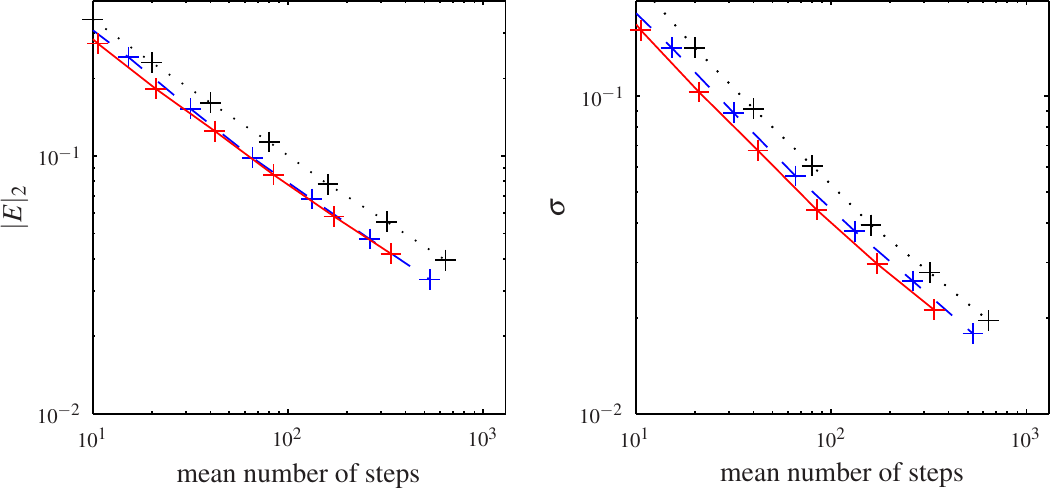} %
  \caption{For \cref{eq:test5}, (left) the sample mean of the
    maximum relative error $|E|_2$ and (right) standard deviation
    $\sigma$ of the error $E$ (see \cref{eq:error}) against the
    mean number of steps taken (based on 5000 samples). We plot
    ({\color{blue}dashed}) Adaptive-I with $\alpha=0.5$,
    ({\color{red}solid}) Adaptive-II with $\alpha=0.9$, and
    (dotted) fixed time-stepping {\EM}
    method.}\label[figure]{fig:a3}
\end{figure}

\subsection{Experiments with adaptive algorithms}\label{ss:exps}
We now test  Adaptive-I and Adaptive-II
with the {\EM} method using the following initial value problems:
\begin{equation}
  \label{eq:test3}
  dy=0.1 y\,dt + 1.2 y\,dW(t), \qquad y(0)=1
\end{equation}
and
\begin{equation}%
  \label{eq:test5}
  dy=1.5 y\,dt + 2.4 y\,dW(t), \qquad y(0)=1.
\end{equation}
Here, $d=m=1$ and $W(t)$ is a standard Brownian motion.
We integrate both equations numerically with the {\EM} method on the
interval $[0,T]=[0,1]$ and compare the result with the exact
solution (geometric Brownian motion). In this example, the drift
and diffusion functions $g_j(y)$ are linear and unbounded. In the
experiments, we replace $\norm{q_{ij}(y_k)}$ in
\cref{eq:bdc,eq:bdalt} with $\min\Bp{\norm{q_{ij}(y_k)},100}$;
this ensures the time steps do not become too small.

For Adaptive-I, time steps and Brownian increments are generated
using the method of \citep{MR1722281}. For Adaptive-II, we'd like %
to sample from an exit time problem on the domain shown in
\cref{fig:1}. We use the following approximate algorithm to
sample from this distribution: Let $R$ denote
the shaded region in \cref{fig:1} for a given $\alpha$ and $h$.
Choose a parameter $\beta>0$.
\begin{enumerate}[1)]
\item Let $(\tau^0, x^0)=(0,0)$ and $k=0$.
\item\label{this} Choose the largest $a_1$ such that $\{\tau^k\}\times
  [x^k-a_1,x^k+a_1]\subset R$, and then the largest $a_0$ such
  that $[\tau^k,\tau^k+a_0]\times [x^k-a_1,x^k+a_1]\subset R$.
\item Use the algorithm of \citep{MR1722281} to find the first exit point
  $(\tau,W(\tau))$ from $[0,a_0]\times
  [-a_1,a_1]$ of the process $(t,W(t))$.
\item Let $\tau^{k+1}=\tau^k+\tau$ and $ x^{k+1}=
  x^k+W(\tau)$. If $\tau < \beta h$,  stop and output
  $(\tau^{k+1},x^{k+1})$. Otherwise, increase $k$ and go to \ref{this}).
\end{enumerate}
The output can be used as the time step and Brownian increment,
which always belongs to the region $R$ but is unlikely to be an
exit point. We apply this method with $\beta=1/10$ to implement
Adaptive-II and the resulting distribution of steps taken is
shown in \cref{fig:c}.

We compute the maximum relative error on the partition ${\cal
  T}$, defined by
\begin{equation}\label{eq:error}
  E%
   \coloneq \frac{1}{Z}%
  \max_{\tau_n\in {\cal T}}
  |y_n-y(\tau_n)|, \qquad%
  \text{ where }%
  Z
   \coloneq \max_{\tau_n\in \mathcal{T}}|y(\tau_n)|.
\end{equation} 
For $M$ \iid samples $E_1,\dots,E_M$ of $E$, we plot 
\[
|E|_2%
 \coloneq \pp{\sum_{j=1}^M E_j^2}^{1/2}
\]
and the standard deviation $\sigma$ of $E_1,\dots,E_M$ against
the mean number of steps, using $M=5000$ samples in
\cref{fig:a3,fig:a4}. The adaptive methods outperform the fixed
time-stepping method in terms of mean error (when taken with the
same mean number of steps) and a 20\% reduction in error is
observed using either adaptive algorithm. Furthermore, the
adaptive methods, especially Adaptive-II, produce a narrower range of errors, as we find
the standard deviation of the errors is smaller. Whilst this is
encouraging, we are discounting the extra time involved in
sampling the bounded diffusion and, as this algorithm is slow
compared to sampling a Gaussian increment, the adaptive methods
are not yet fully practical.
 
Two further plots are shown for \cref{eq:test3}.  \cref{fig:b}
shows a scatter plot of the errors $E$ against the number of
steps taken for the three methods.  \cref{fig:c} shows the mean
and standard deviation of the number of steps used in the
experiments for a set of $h$ values.

\begin{figure}%
  \centering
  \includegraphics{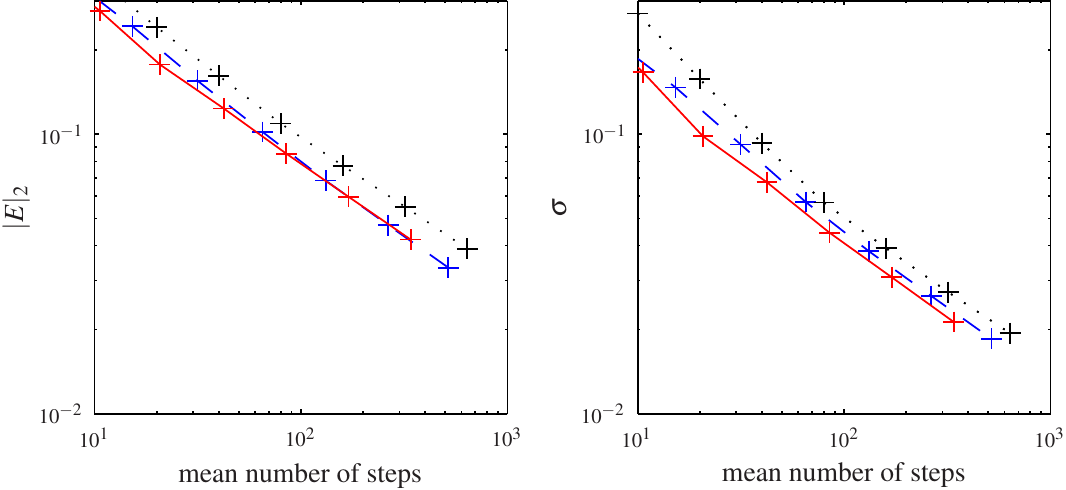} 
  \caption{For \cref{eq:test3}, (left) $|E|_2$ and (right)
    standard deviation $\sigma$ of $E$ against the mean number of
    steps taken. We plot ({\color{blue}dashed}) Adaptive-I with
    $\alpha=0.5$, ({\color{red}solid}) Adaptive-II with
    $\alpha=0.9$, and (dotted) fixed time-stepping {\EM}
    method.} \label[figure]{fig:a4}%
\end{figure}

\begin{figure}\centering
  \centering
  \includegraphics{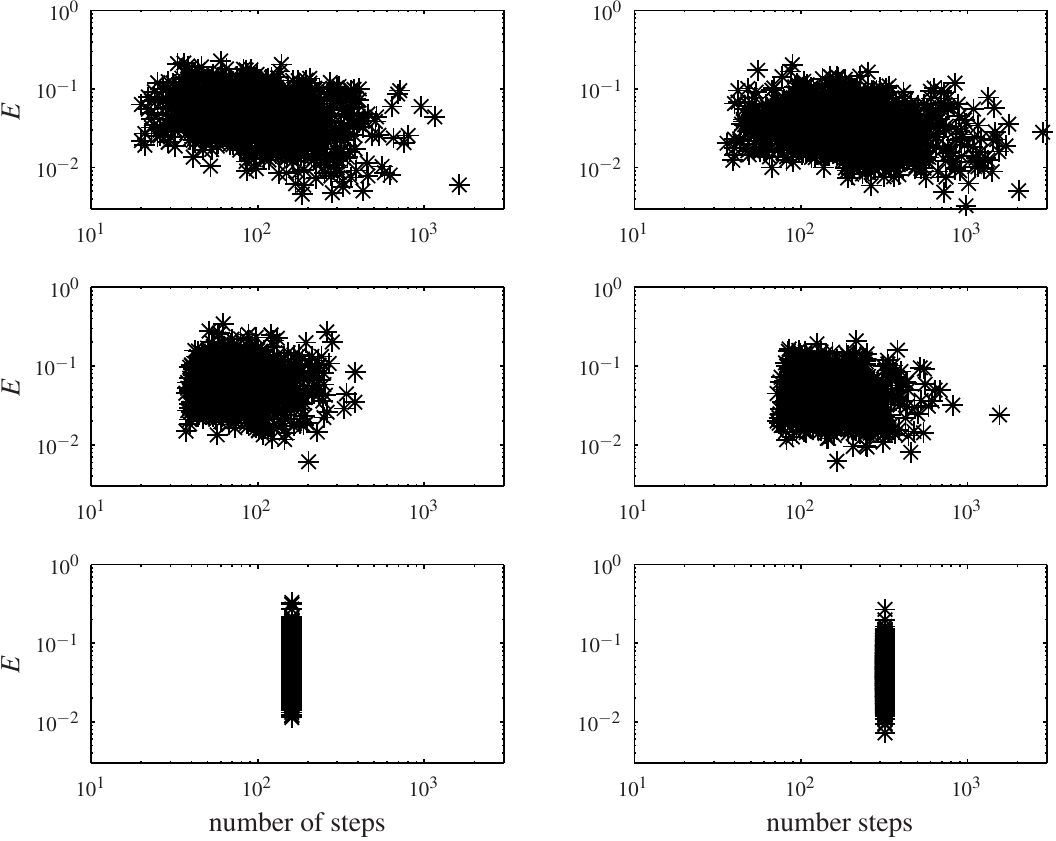} 
  \caption{Plots of the relative error $E$ against number of
    steps taken $N$ for 100 samples of the approximation to
    \cref{eq:test3} for (top) Adaptive-I with $\alpha=0.5$,
    (middle) Adaptive-II with $\alpha=0.9$, and (bottom) fixed
    time-stepping {\EM}.}\label{fig:b}
\end{figure}

\begin{figure}\centering
  \centering 
  \includegraphics{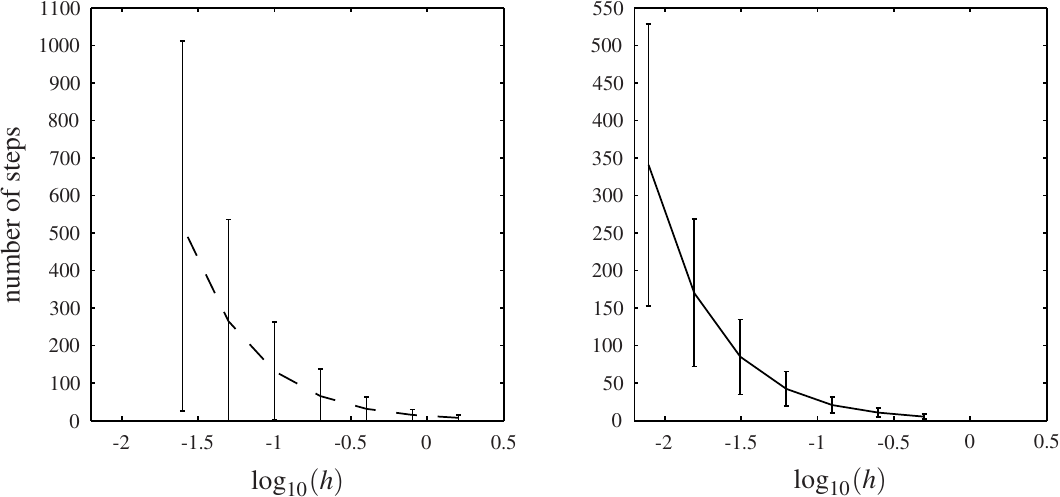} 
  \caption{Mean number of steps with error bars indicating one
    standard deviation for (left) Adaptive-I with $\alpha=0.5$
    and (right) Adaptive-II with $\alpha=0.9$ applied to
    \cref{eq:test3}.}\label{fig:c}
\end{figure}

\section{Conclusion}

We presented a new proof of pathwise convergence for
numerical approximation of SDEs, which avoids the need to
prove the $p$th-mean error converges at a polynomial
rate. This requires two pathwise assumptions: first, a
Lipschitz type assumption of the deviation of the sample
path of the solution from its initial data and, second, a
pathwise bound on the sum of the truncation errors. The
proof of the pathwise-convergence theorem uses no
probabilistic arguments. We showed how to apply the theorem
to the {\EM} method with fixed time-stepping. We also 
introduced two adaptive time-stepping methods, motivated by
the truncation error sum condition. To complete these proofs
and verify the assumptions of the pathwise-convergence
theorem, we do use probabilistic arguments and estimate
moments and apply the Borel--Cantelli lemma.  The main
advantage of this approach is that more detailed constants
are gained and we are able to find a tighter error bound for
the adaptive methods. Computations with SDEs usually work by
computing paths, even if the quantity of interest is an
average or exit time, and the pathwise condition on the
truncation error sum is a convenient framework for studying adaptive
methods.

Some experiments were shown with geometric Brownian motion using
a method of \citep{MR1722281} for bounded diffusions to generate
the appropriate time steps and Brownian increments. For a given
mean number of time steps, the errors are smaller and show less
variation than for fixed time-stepping. This extra accuracy,
which does not change the rate of convergence, requires sampling
bounded diffusions rather than diffusions over fixed time-steps
(\iid Gaussian random variables) and this makes the algorithm
expensive to implement. Finding a fast method for sampling
bounded diffusions is a key point to be addressed in future
research.
\appendix

\section{Useful results}
The first result concerns the regularity of sample paths and
gives a large class of SDEs where \cref{ass:4.2} holds.
\begin{lemma}%
  \label[lemma]{lemma:quote_me}%
  If $\vec g_j\in \Cb^3(\real^d,\real^d)$, then the solution
  $y(t;s,\vec z)$ of \cref{eq:sode} satisfies \cref{ass:4.2}.
\end{lemma}
\begin{proof}
Fix $\epsilon>0$.
  Let $\vec x(t;s,\vec z) \coloneq \vec y(t;s,\vec z)-\vec z$ for $\vec
  z\in\real^d$ and $0\le s,t\le T$.  Then, $\vec x(t;s,z)$
  satisfies
  \[
  d\vec x%
  =d\vec y%
  =\sum_{j=0}^m g_j(\vec y)\,dW^j(t)%
  =\sum_{j=0}^m g_j(\vec x+\vec z)\,dW^j(t),
  \qquad \vec x(s;s,\vec z)=0
  \]
  and $\vec x(t; s,\vec z^i)$ for $i=1,2$ are two solutions of an
  SDE with modified vector fields $\vec V_j^i(\vec q)=\vec
  g_j(\vec q+\vec z^i)$. Then,  \citep[Theorem
  10.26]{friz_multidimensional_2010} 
  gives
  \[
  \norm{\vec x(t;s,\vec z^1)-\vec x(t;s,\vec z^2)}%
  \le C|t-s|^{1/2-\epsilon} \max_j\norm{V^1_j-V^2_j}_{\Cb^2}
  \]
  for some $C$ dependent on the sample path of $W^j(t)$. Because
  $g_j\in \Cb^3(\real^d,\real^d)$, it is clear that
  $\norm{V^1_j-V^2_j}_{\Cb^2} \le \norm{g_j}_{\Cb^3}\norm{\vec
    z^1-\vec z^2}$ and hence 
  \[
  \norm{\pp{\vec y(t;s,\vec z^1)-\vec z^1}%
    -\pp{\vec y(t;s,\vec z^2)-\vec z^2} }%
  \le C |t-s|^{1/2-\epsilon} \norm{\vec z^1-\vec z^2},%
  \]
  for all $\vec z^1, \vec z^2\in\real^d$ and $0\le s,t\le T$.
  This is \cref{ass:4.2}.
\end{proof}

For a stochastic process $Y(t)$, the modulus of continuity
\[
\omega_Y(h)%
 \coloneq \sup_{t,s\in [0,T],\; |t-s|\le h} \norm{ \vec Y(t)-\vec Y(s) },\qquad
h>0.
\]
The following result examines the modulus of continuity when $Y$
is an \Ito integral and is applied to the truncation
error in order to derive Assumption~\ref{ass:newx}.
\begin{proposition}%
  \label[proposition]{prop:recent}
  Let
  \[
  \vec Y(t)%
  = \int_0^t \vec q_0(s)\, ds%
  +\sum_{j=1}^m \int_0^t \vec q_j(s)\, dW^j(s)
  \] 
  for predictable processes $\{\vec q_j(s)\colon
  s\in[0,T]\}$. Consider random variables $\zeta,\xi$ such that,
  for some $p\ge 1$ and $K,\mu>0$,
  \[
  \mean{|\zeta|^p}\le K\quad \text{ and }\quad%
  \mean{|\xi|^{p/2+\mu}}\le K,
  \] and such that, for $0\le s<t\le T$,
  \[
  \int_s^t \norm{\vec q_0(r)} \,dr%
  \le \zeta\;\sqrt{|t-s|\log(2T/|t-s|)},\qquad
  \sum_{j=1}^m  \int_s^t \norm{\vec q_j(r)}^2 \,dr%
  \le \xi\;|t-s|.
  \]
  Then, for any $0< \epsilon<1/2$, there exists a random
  variable $C_\epsilon\in L^p(\Omega)$  such that 
  \[
  \omega_Y(h)
  \le C_\epsilon h^{1/2-\epsilon}\qquad %
  \text{for all $s,t\in[0,T]$}
  \]
  and $\norm{C_\epsilon}_{L^p(\Omega)}$ depends only on
  $\epsilon, \mu, p, K, T$ and is independent of $q_j$.
\end{proposition}
\begin{proof}
  Under these conditions,  \citet{fischer_moments_2009} tell us
  there exists $C>0$ depending only on $p, \mu,$ and $K$ such that
  \[
  \norm{\omega_Y(h)}_{L^p(\Omega)}%
  \le C \sqrt{h \log (2T/h)}.
  \]
  Then, for $\epsilon>0$, there exists a $C_\epsilon\in
  L^p(\Omega)$ such that $\omega_Y(h) \le C_\epsilon
  h^{1/2-\epsilon}$ and hence
  \[
  |Y(s)-Y(t)|%
  \le C_\epsilon |s-t|^{1/2-\epsilon},\qquad %
  s,t\in[0,T].
  \]
  The $L^p(\Omega)$ norm of $C_\epsilon$ is bounded uniformly in
  $\epsilon, \mu,$ and $K$, as required.
\end{proof}

The following is a consequence of the well-known Azuma inequality
\citep{MR0221571,MR2807365}.
\begin{lemma}\label[lemma]{azuma}
  Let $Y_i$ be a sequence of scalar random variables with
  $|Y_i|\le \Ybar$ a.s for a constant $\Ybar$.  Suppose that
  $\mean{Y_i | Y_0,\dots,Y_{i-1}}=0$ a.s. for
  $i=1,\dots,n$. Let $S_n=Y_1+\dots+Y_n$. Then, for $\lambda>0$,
  \[
  \prob{|S_n-S_k| \ge \lambda} %
  \le 2 \exp\pp{-\frac{\lambda^2}{2 (n-k) \Ybar^2}}.
  \]
\end{lemma}



\end{document}